\documentclass[11pt]{article}
\usepackage{srcltx}
\usepackage[colorlinks=false,bookmarks=true]{hyperref}
\usepackage{amsfonts,amsmath,amsthm,amssymb,latexsym,enumerate,subfigure,appendix}
\usepackage[latin2]{inputenc}
\usepackage[english]{babel}
\usepackage{pgf}
\usepackage{url}

\makeatletter

\newdimen\mainfontsize \mainfontsize=1\@ptsize pt
\topskip=\mainfontsize
  \@maxdepth=\maxdepth

\makeatother

\newtheorem{thm}{Theorem}[section]
\newtheorem{defn}[thm]{Definition}

\newtheorem{lem}[thm]{Lemma}

\newtheorem{rem}[thm]{Remark}
\newtheorem{ass}[thm]{Assumption}
\newtheorem{alg}[thm]{Algorithm}

\newcommand{\be}{{\mathbb E}}

\newcommand{\cf}{{\cal F}}
\newcommand{\cp}{{\cal P}}

\newcommand{\cg}{{\cal G}}

\newcommand{\cT}{{\cal T}}

\newcommand{\br}{\mathbb R}

\newcommand{\mailto}[1]{\href{mailto:#1}{#1}}
\newcommand{\Vup}{V^{\uparrow}}
\newcommand{\martFuns}{\mathbf{M}} % space of M_0-valued functions

\makeatletter
\@addtoreset{equation}{section}

\makeatother

\begin{document}

\begin{center}
{\Large{\bf Monte Carlo methods via a dual approach for some \\[.1in] discrete time stochastic control problems}}
\end{center}
\vspace*{.1 true in}
\begin{center}
{\large L. G.~Gyurk\'o}\footnote{Mathematical Institute, University of Oxford, 24-29 St Giles, Oxford OX1 3LB, UK. \newline
E-mail:  \mailto{gyurko@maths.ox.ac.uk}},
\quad
{\large B. M.~Hambly}\footnote{Mathematical Institute, University of Oxford,
24-29 St Giles, Oxford OX1 3LB, UK.\newline
~\mbox{ } E-mail:  \mailto{hambly@maths.ox.ac.uk}},
\quad
{\large J. H.~Witte}\footnote{Mathematical Institute, University of Oxford, 24-29 St Giles, Oxford OX1 3LB, UK. \newline
E-mail: \mailto{witte@maths.ox.ac.uk}}

\date{\today}
\end{center}

\begin{abstract}
We consider a class of discrete time stochastic control problems motivated by some
financial applications. We use a pathwise stochastic control approach to provide a
dual formulation of the problem. This enables us to develop a numerical technique
for obtaining an estimate of the value function which improves on purely regression based
methods. We demonstrate the competitiveness of the method on the example of a gas storage valuation problem.
\end{abstract}

\bibliographystyle{plain}

%****************************************************************
%****************************************************************
\section{Introduction}
%****************************************************************
%****************************************************************

The numerical pricing of options with early exercise features, such as American options, is a challenging problem, 
especially when the dimension of the underlying asset increases. There is a large body of literature which discusses 
this problem from different 
points of view, beginning with techniques aimed at solving the dynamic programming problem using trees or
the associated Hamilton-Jacobi-Bellman equation. Over the past decade, there has been a lot of activity in developing
Monte Carlo techniques for such optimal stopping problems. The most popular have been basis function regression methods
initially proposed in \cite{LonSchw} and \cite{VanRoyTsits}. If these methods are used to provide an approximate optimal
exercise strategy, they naturally provide 
lower bounds for prices. Thus they were soon followed by dual methods \cite{Rog,HauKog} designed to find upper 
bounds. An account of these methods can be found in \cite{glasserman}.

Following on from the development of dual methods for American options, there has been a strand of research 
extending these ideas to multiple optimal stopping problems, which correspond to options with multiple 
exercise features \cite{MeiHam}. The dual method proceeds via the idea of pathwise optimization, which originated 
in \cite{DavKar}. This pathwise optimization method was developed in a general setting in \cite{Rog1} where 
it was applied to more general stochastic control problems.

In this paper, our aim is to consider a subclass of such stochastic control problems for which we can develop a 
relatively simple dual approach and which leads to numerical algorithms for the efficient computation of the 
value function.

We were originally motivated by option pricing problems in the electricity market. In that setting contracts such as 
swing options give the holder certain rights to exercise variable amounts through the lifetime of the contract. 
The dual approach, initiated in \cite{MeiHam}, used a simplistic swing contract
in which a single exercise was allowed on each day, with the total number of exercise rights over the lifetime of the 
contract constrained. In \cite{AleHam, Ben}, this was extended to multiple discrete exercises on a given day. Other 
recent developments have seen a move to continuous time \cite{Ben2} and a `pure martingale' dual formulation
of the problem \cite{Sch}.

Our first aim in this paper is to provide a more general formulation of the dual problem 
in discrete time which allows exercise of continuous amounts and contains the `pure martingale' approach. 

Our second aim is to provide a useful numerical approach to this type of problem. Having moved beyond the 
multiple optimal stopping problem to a more general stochastic control formulation, the space of controls is
now potentially of dimension greater than one, and consequently more difficult to handle. Instead of a purely binary 
decision (or at most a finite set of decisions) at each time point, we have the possibility of choosing from a 
Euclidean space (in the electricity context, this is exercising a real amount corresponding to a volume of power). 
Our dual formulation of the problem leads naturally to an upper bound on the value function. We develop a technique 
based on being given an a priori estimate for the value function, say typically an estimate obtained via basis
function regression, and converting this to an improved estimate via the dual. 

In order to produce the a priori estimate, the method uses least squares regression and Monte 
Carlo techniques, an extension of the approach due to Longstaff-Schwarz \cite{LonSchw} and Tsitsiklis and van Roy
\cite{VanRoyTsits}; we use test functions that depend on both the underlying factor and the control value.
We note that this idea has been considered by Boogert and Jong \cite{BooJon}; however, Boogert and Jong did not 
develop the extended regression based method in detail, but worked with regression depending only on the 
underlying factor for several discrete values of the control. Belomestny et al. \cite{BelKolSch} have also 
developed a family of least squares regression and Monte-Carlo based numerical algorithms. The algorithm in
\cite{BelKolSch} can be applied to more general discrete time control problems than the ones we consider in 
this paper. However, as in \cite{BooJon}, Belomestny et al. regress the conditional expectation arising in the 
dynamic programming principle using test functions depending on the underlying factor only.
When applied to the same control problem, with the right choice of test functions and grid in the space of underlying 
factor and control, we found that our extended regression based method performs better than the method in \cite{BooJon} 
or the method in \cite{BelKolSch}, especially when the control is high dimensional.

The a priori estimate is used as an input to the dual formulation based upper bound. The implementation of the dual
estimate requires the numerical solution of several independent deterministic optimal control problems. We note 
that these control problems can be solved simultaneously, and, hence, it is well suited for a parallel implementation.

As an application, we will focus on one example in this paper, namely natural gas storage valuation. The owner 
of a natural gas storage facility is faced with an optimal control problem in order to maximize the return from 
running the facility. The demand for natural gas is seasonal with high demand and prices in the winter, and low 
demand in the summer. The operator of a facility will want to buy and store gas when it is cheaper over the 
summer, and then sell gas into the market when the price is higher in the winter. The operation of the facility 
is thus a control problem where, on a given day, the operator has the decision to inject or produce a 
volume of gas, given the current price of gas. Thus, we have the set up of a stochastic control problem of the 
type we consider here. We chose the particular gas storage problem as a numerical example in order to compare 
the results of our probabilistic approach to the results of the partial differential equation based methods 
(cf.\,\cite{ChenFors,ThoDavRas}). In general, we expect the probabilistic approach to perform better than the 
PDE methods when the dimension of the underlying factor and/or the dimension of the control is high.

Our numerical example demonstrates that the dual formulation based upper bound is sharper than the one 
we get from the a priori estimate at comparable computational expense. This empirical observation justifies 
the potential benefit of computing the dual formulation based estimate in practice.

The outline of the paper is as follows. We will begin with the setup for the problem in Section~2 and follow this 
with the dual formulation in Section~3. We obtain our main representation in Theorem~\ref{_thm_dual_formulation}, 
and then derive a 
version which can be used for the Monte Carlo based numerical technique in Lemma~\ref{lem_UpperLowerDifference}. 
We follow this with a discussion of the numerical technique itself in Section~4. Finally, we apply the approach to 
the gas storage problem in the last section.

%****************************************************************
%****************************************************************
\section{Discrete time decoupled stochastic control problems}
%****************************************************************
%****************************************************************

We consider an economy in discrete time defined up to a finite time horizon $T$. We assume a 
financial market described by the filtered probability space $(\Omega, \mathcal{F},
(\mathcal{F}_t)_{t\in\cT},\mathbb{P})$, where $\cT=\{0,1,\dots,T\}$. We take $(X_t)_{t\in\cT}$ to be 
an $\mathbb{R}^d$-valued discrete time Markov chain representing the price of the
underlying assets and any other variables that affect the dynamics
of the underlyings. We assume that the filtration $(\mathcal{F}_t)_{t\in\cT}$ is generated by $X$. 
Moreover, we assume that $\mathbb{P}$ is a risk neutral
pricing measure, and write ${\mathbb E}_t(X)={\mathbb E}(X|\mathcal{F}_t)$ for any random
variable $X$ on our probability space. Throughout the paper, we will assume that interest rates are 0.

We phrase our problem in the language of options, even though it is a standard stochastic control problem
of maximizing a reward obtained from a randomly evolving system. The payoff of the option (or the 
reward for the position) $H_t(h,X_t)$ 
at time $t\in\mathcal{T}$ is a function of the underlying $(X_t)_{t\in\cT}$ and 
the exercise amount $h\in \br^{l}$, which is chosen by the holder of the option subject to certain constraints.
The problems that we consider are decoupled in that the decision at time $t$ regarding $h_t$ has no impact on
the evolution of the underlying state of the economy $X_s$ for $s>t$.

The set of \emph{ admissible exercise decisions } available at a given time is defined by a (set-valued) function 
$K$ on $\cT\times \br^l \times \br^d$ that takes values in the set of subsets of $\br^l$. The control process
$(y_t)_{t\in\cT}\subset\br^{l}$ is defined by $t\ge 0$, $y_{t+1}:=y_t-h_t$, for the exercise amount $h_t$.
We will write $K_t(y_t,X_t(\omega))$, or, if needed, $K(t,y_t,X_t(\omega))$, for the given set of admissible 
exercise decisions depending on $t$, the state of the underlyings $X_t$, and 
the value of the control process $y_t$. The initial value $y_0$ and the constraints 
$K_t(\cdot,\cdot)$ for $t\in\mathcal{T}$ are determined by the option contract.

% In order to define the value function $V^K_t(\cdot,\cdot)$ we require a little notation.

\begin{defn}[$K$-admissible exercise policy]\label{_def_KAdmissiblePolicy}
A policy, or exercise strategy, $\pi=(h_t,\dots,h_T)$ is a $K$-\emph{admissible exercise policy} on 
$\{t,\dots,T\}$ started at $y$ if it satisfies all of the following properties.
\begin{enumerate}[(i)]
\item $h_s$ is $\mathcal{F}_s$-measurable for $s=t,\dots, T$, 
%\item the payoff $H(h_s,X_s)$ is integrable for each $s$ for all exercise amounts $h_s$,
%the expectation on the right-hand-side of \eqref{_eq_PolicyValue} corresponding to $\pi$ is well defined, 
\item and $h_s(\omega)\in K_s(y_s,X_s(\omega))$ for all $s=t,\dots,T$ and for all $\omega$ in a set of probability one,
\end{enumerate}
where $y_s$ is defined recursively by $y_t=y$ and $y_{s+1}=y_{s}-h_{s}$ for $s=t,\dots,T-1$. 

The set of such policies is denoted by $\cp_{K,y,t}$.
\end{defn}

Thus, a $K$-admissible exercise policy $\pi$ on the time-set $\{t,\dots,T\}$ is defined by the $\cf$-adapted process 
$(h_s)_{s=t,\dots,T}$ describing the exercise decisions at times $t,\dots,T$, and the value of such an exercise policy 
$\pi$ at time $t$ is given by
\begin{equation}
V^{K,\pi}_t(y_t,x) = \be\left[\sum_{s=t}^T H_s(h_s,X_s) | X_t=x\right]. 
\label{_eq_PolicyValue}
\end{equation}
In the particular examples considered in this paper, the set $K_t$ will be a line segment in $\br$ or a quadrant of $\br^2$.

We are now in a position to define the value function $V^{K,*}_t(y_t,X_t)$ 
at time $t$ of the option satisfying the constraints $K$.

% 
% We use the notation $V^K_t(y_t,X_t)$ for the price (value function) 
% (see definition \ref{_def_ValueFunction}). 
% For simplicity of notation, we assume that $H_t$ (respectively $V^K_t$) denotes the payoff (option value) discounted to time $0$. 

%Given a family of set-valued functions $\{K_t(\cdot,\cdot),t\in\cT\}$ describing the set of admissible exercise decisions, we use the symbol $\cp_{K,y,t}$ to denote the set of \emph{admissible exercise policies} on $\{t,\dots,T\}$, satisfying $y_t=y$ and $h_s\in K_s(y_s,X_s)$ for $s=t,\dots,T$. 

%*************************************************************
\begin{defn}\label{_def_ValueFunction}
We define the value function to be
\[ V^{K,*}_t(y,x) 
= \sup_{\pi\in \cp_{K,y,t}} V^{K,\pi}_t(y,x) 
= \sup_{\pi\in \cp_{K,y,t}}\be\left[\sum_{s=t}^T H_s(h_s,X_s)|X_t=x\right]
, \ (y,x)\in\br^l\times\br^d.
\]
\end{defn}
%*************************************************************

For simplicity, we make the following assumption. 
\begin{ass}\label{ass-dpp}
There exists a set $Y_0^K$ of initial control values 
and a bound $C$ such that
\[ \be[|H_s(h,X_s)|] <C \;\;\forall s \in \cT, h \in K_s(y,X_s), \]
for all $(y,X_s)$ reachable at time $s$ from $Y_0^K \times \{X_0(\omega)|\omega\in\Omega\}$ by a $K$-admissible 
policy. 
\end{ass}
This is enough to ensure the existence of the value function and the dynamic programming principle. Weaker
assumptions which guarantee existence would be possible, but are not the focus of this paper.

%Let $Y_0^K = \{y_0: V^{K,*}_0(y_0,X_0)$ exists almost surely $\}\subset\mathbb{R}^l$ be the set of 
%valid starting values for the control. Throughout the paper, we restrict our focus to initial control values in $Y_0^K$. 

%is assumed to be a possibly infinite subset of $\br^l$. 

In order to indicate the type of problems that fit into this framework, we give three examples. In the final section, we 
will focus on the first.

\vspace{.1in}
\noindent \underline{Gas storage valuation:}\\[.1in]
Natural gas storage valuation and optimal operation can be formulated as an option contract as described above. 
In particular, let $X_t\in\br$ denote the spot price of natural gas at time $t$, and let $y_t\in\br$ denote the amount of 
gas stored in the facility at time $t$. In that case, with $h_t$ denoting the change of level between $t$ and $t+1$, the 
payoff is defined by
\[
H_t(h_t,X_t)=h_tX_t.
\]
More accurate models may also take into account the loss of gas occurring at injection. 

At time $t$, the set $K_t(y_t,X_t)$ is determined by the maximum and minimum capacity of the gas storage facility, 
and by the injection/production rate depending on the stored amount $X_t$. A continuous time description of this 
problem was given in \cite{LudCar} and in \cite{ThoDavRas}. In section \ref{_sec_gas_storage}, we present 
a time-discretized version. 

\vspace{.1in}
\noindent\underline{Swing option pricing:}\\[.1in]
In the electricity market, a swing option enables the holder to protect themselves
against the risk of price spikes if they are exposed to the spot price of electricity $X$.
The simplest versions give their holder the
right, for a specified
period of time, to purchase each day (on- or off- peak time)  electricity at a fixed price $K$ (strike price). In this case, the payoff is that of a 
call $(X_t-K)^+$. When exercising a swing option at a time $t$, the amount purchased may vary (or
swing) between a minimum volume, $m_t$, and a maximum volume, $M_t$, while the total quantity
%There are alternatives such as
%when the swing contract allows the holder either to buy or to sell
%a certain quantity of electricity on a given day.  The total quantity of electricity
purchased for the period must remain within minimum $\bar{m}$ and maximum $\bar{M}$ volume levels.
Thus, we have
\[ H_t(h_t,X_t) = h_t\max(X_t-K,0), \]
with $m_t\leq h_t \leq M_t$ and $\bar{m} \leq \sum_{t=0}^T h_t \leq \bar{M}$. The set $K_t$ is the line segment
determined by these constraints.

\vspace{.1in}
\noindent\underline{Optimal liquidation:}\\[.1in]
A similar approach can be pursued to model the optimal liquidation of a portfolio of dependent assets. 
Let $X_t=(X^1_t,X^2_t)\in\br^2$ denote the value of two bonds by the same issuer with different 
issue dates. We assume that, whenever $X^1$ is traded, there is a temporary price impact on both, 
$X^1$ and $X^2$, referred to as a \emph{multi-asset price impact}. For example, if
$h_t=(h^1_t,h^2_t)\in\br^+\times\br^+$ denotes the quantities sold of the bonds 
$X^1$ and $X^2$, respectively, the payoff is
%\[
%%H_t(h_t,X_t)=h^1_t\frac{1}{\sqrt{h^1_t+h^2_t}}X^1_t+h^2_t\frac{1}{\sqrt{h^1_t+h^2_t}}X^2_t.
%H_t(h_t,X_t)=h^1_tX^1_t\frac{1}{(1+\alpha_{11})^{h^1_t}(1+\alpha_{21})^{h^2_t}}+
%h^2_tX^2_t\frac{1}{(1+\alpha_{12})^{h^1_t}(1+\alpha_{22})^{h^2_t}},
%\]
\[
H_t(h_t,X_t)=h_t^1X_t^1+h_t^2X_t^2-(h_t^T\Lambda h_t)^{\beta}
\] 
%where $\alpha_{ij}>0$ describes the price impact on asset $j$ when asset $i$ is sold. 
for some $\beta>1/2$ and some matrix $\Lambda\in\br^{2\times 2}$. The reader is referred to 
\cite{torsten} for further details. The one-dimensional case is considered in \cite{AleHam2}.

In this setting, the set $K_t$ is determined by the total volume of each bond that we hold and wish to liquidate, and
is hence a subset of $\mathbb{R}^2$.

Many models (see for example \cite{AlmgrenChriss})  consider modelling the permanent price impact of trades 
on top of incorporating the temporary impact.  
In the presence of permanent impact, the optimal liquidation problem is typically formulated as a coupled stochastic 
control problem. The approach based on coupled control problems falls beyond the scope of this paper.  For certain 
models of permanent price impact (e.g. non-resilient impact), the a priori estimate presented in 
Section~\ref{_seq_num_lower_bound} can be easily adapted. Deriving a dual formulation, however, is 
less straightforward. 

%\cite{Rog1}

%****************************************************************
%****************************************************************
\section{Dual formulation}
%****************************************************************
%****************************************************************

Definition \ref{_def_ValueFunction} represents the value of the option as the supremum over the set of 
admissible exercise policies. We now develop a dual for this problem that represents the option value as an infimum 
over a space of martingale-valued functions. Let $\martFuns$ denote the space of functions defined on $\br^l$ 
and taking values in the space ${\cal M}_0$ of martingales which are adapted to the filtration 
$(\mathcal{F}_t)_{t\in\cT}$ and null at time $0$. For $M\in\martFuns$, $y\in\br^l$, $t\in\cT$, $M^y_t$ 
denotes the time-$t$ value of $M(y)\in{\cal M}_0$. 

%*************************************************************
\begin{thm}\label{_thm_dual_formulation}
Let $K$ be a function defined on $\cT\times \br^l \times \br^d$ and taking values in the set of subsets of $\br^l$. Then, for all $y\in Y_0^K$, the value $V_0^{K,*}(y,X_0)$ of the option at time $0$ almost surely satisfies 
the following. 
%The value of the option at time $0$ can be represented as
\begin{equation}
V_0^{K,*}(y,x) 
= \inf_{M\in \martFuns} \be\left[\sup_{\pi\in \cp_{K,y,0}} \sum_{t=0}^{T-1}( H_t(h_t,X_t) - M^{y_{t+1}}_{t+1}
+ M^{y_{t+1}}_t ) + H_T(h_T,X_T)\bigg|X_0=x\right].
\label{_eq_dual_thm}
\end{equation}
Moreover,
\begin{equation}
V_0^{K,*}(y,x) 
= \be\left[\sup_{\pi\in \cp_{K,y,0}} \sum_{t=0}^{T-1}( H_t(h_t,X_t) - M^{K,y_{t+1}}_{t+1}
+ M^{K,y_{t+1}}_t ) + H_T(h_T,X_T)\bigg|X_0=x\right],
\label{_eq_dual_thm}
\end{equation}
where $M^{K,y}\in\mathcal{M}_0$ such that 
\begin{equation*}
M^{K,y}_{t+1}:=M^{K,y}_t + V_{t+1}^{K,*}(y_,X_{t+1})-
\be_t\left[V_{t+1}^{K,*}(y,X_{t+1})\right]. 
\end{equation*}

\end{thm}
%*************************************************************

%*************************************************************
\begin{proof}
We follow a similar approach to that of Rogers \cite{Rog1}. We have
\begin{align*}
V_0^{K,*}(y,x) 
&= \sup_{\pi\in \cp_{K,y,0}} \be\left[\sum_{s=0}^T H_s(h_s,X_s)\bigg|X_0=x\right] \\
&= \sup_{\pi\in \cp_{K,y,0}}\be\left[\sum_{s=0}^{T-1}\big( H_s(h_s,X_s)-M^{y_{s+1}}_{s+1}
+M^{y_{s+1}}_s\big) + H_T(h_T,X_T)\bigg |X_0=x\right] \\
&\leq  \be\left[ \left\{\sup_{\pi\in \cp_{K,y,0}}\sum_{s=0}^{T-1}\big( H_{s}(h_s,X_s)-M^{y_{s+1}}_{s+1}
+M^{y_{s+1}}_s\big) + H_T(h_T,X_T)\right\}\bigg |X_0=x\right].
\end{align*}
As this holds for all martingales $M^w$, we then have
\[ 
V_0^{K,*}(y,x) 
\leq \inf_{M\in\martFuns} \be\left[\sup_{\pi\in \cp_{K,y,0}} \sum_{s=0}^{T-1}\big( H_s(h_s,X_s) - M^{y_{s+1}}_{s+1}
+M^{y_{s+1}}_s \big) + H_T(h_T,X_T)\bigg|X_0=x\right]. 
\]

To see that the inequality holds the other way around, we consider a particular family of martingales. The one that we
take is $\{M_t^{K,y_t},t\in\cT\setminus T\}$ from the Doob decomposition of the value function. Thus, its increments are given by
\[ 
\Delta M^{K,y_{t+1}}_t = M^{K,y_{t+1}}_{t+1}-M^{K,y_{t+1}}_t = V_{t+1}^{K,*}(y_{t+1},X_{t+1})-
\be_t\left[V_{t+1}^{K,*}(y_{t+1},X_{t+1})\right]. 
\]
Using this martingale, we have
\begin{align*}
&  \inf_{M\in\martFuns} \be\left[\sup_{\pi\in \cp_{K,y,0}}\left\{\sum_{s=0}^{T-1}( H_s(h_s,X_s) - M^{y_{s+1}}_{s+1}
+M^{y_{s+1}}_s ) + H_T(h_T,X_T)\right\}\bigg|X_0=x\right] \\
&\qquad \qquad \leq  \be\left[\sup_{\pi\in \cp_{K,y,0}}\left\{\sum_{s=0}^{T-1}( H_s(h_s,X_s) - \Delta M^{K,y_{s+1}}_s) 
+ H_T(h_T,X_T)\right\}\bigg|X_0=x\right] \\
&\qquad\qquad = \be\Bigg[\sup_{\pi\in \cp_{K,y,0}}\Bigg\{\sum_{s=0}^{T-1}\big( H_s(h_s,X_s) - V_{s+1}^{K,*}(y_{s+1},X_{s+1})+\be_s\left[V_{s+1}^{K,*}(y_{s+1},X_{s+1})\right]\big) \\
& \qquad\qquad\qquad\qquad\qquad+ H_T(h_T,X_t)\Bigg\}\bigg|X_0=x\Bigg].
\end{align*}

By the definition of the value function $V^{K,*}_t(\cdot,\cdot)$, for any $(y,x)\in\mathbb{R}^l\times\mathbb{R}^d$, $t\in\mathcal{T}$, and $h\in K_t(y,x)$, we have 
\[ V^{K,*}_t(y,x) 
\ge H_t(h,x)+\sup_{\pi\in \cp_{K,y-h,t+1}}\be\left[\sum_{s=t+1}^T H_s(h_s,X_s)\big|X_t=x\right],
\]
and, therefore,
%From the dynamic programming equations we have
\[ V_t^{K,*}(y,x) \geq H_t(h,x) + \be\left[ V_{t+1}^{K,*}(y-h,X_{t+1})|X_t=x\right]. \]
Hence,
\begin{align*}
&  \inf_{M\in\martFuns} \be\left[\sup_{\pi\in \cp_{K,y,0}}\left\{\sum_{s=0}^{T-1}( H_s(h_s,X_s) - M^{y_{s+1}}_{s+1}
+M^{y_{s+1}}_s ) + H_T(h_T,X_T)\right\}\bigg|X_0=x\right] \\
&\qquad\qquad \leq  \be\left[\sup_{\pi\in \cp_{K,y,0}} \left\{\sum_{s=0}^{T-1}\big( V_s^{K,*}(y_s,X_s) - 
V_{s+1}^{K,*}(y_{s+1},X_{s+1})\big) + H_T(h_T,X_T)\right\}\bigg|X_0=x\right] \\
&\qquad \qquad = V_0^{K,*}(y,x) + \be\left[\sup_{\pi\in \cp_{K,y,0}} \big\{H_T(h_T,X_T)-V_T^{K,*}(h_T,X_T)\big\}\right]. 
\end{align*}
Now, using the fact that at $T$ we must have $V_T^{K,*}(y,x)=H_T(y,x)$, we have
\[ \inf_{M\in\martFuns} \be\left[\sup_{\pi\in \cp_{K,y,0}}\left\{\sum_{s=0}^{T-1}( H_s(h_s,X_s) - M^{y_{s+1}}_{s+1}
+M^{y_{s+1}}_s ) + H_T(h_T,X_T)\right\}\bigg|X_0=x\right] \leq  V_0^{K,*}(y,x)\]
as required. 
\end{proof}
%*************************************************************

%*************************************************************
\begin{rem} {\rm
Consider the specification of the multiple stopping problem in \cite{Sch}.
The payoff function is $H_t(h,x) = hx$ for $t\in\mathcal{T}$, the control satisfies $0<y_0\le T+1$ and 
takes non-negative integer values, and the constraint sets are defined by
\begin{equation*}
K_t(y,x)=K_t(y)=\left\{
\begin{array}{cl}
\{1\} & \text{ if } y \ge T-t+1, \\
\{0,1\} & \text{ if } T-t+1>y>0, \\
\{0\} & \text{ if } y=0.
\end{array}
\right.
\end{equation*}
In this special case, the payoff value is either $0$ or $X_t$.  Hence, \eqref{_eq_dual_thm} simplifies to the following. 
\begin{equation*}
V_0^{K,*}(y_0,x) = \inf_{M^1,\dots,M^k\in\mathcal{M}_0} \be\left[\max_{0\le t_1 < \cdots < t_{y_0} \le T}
\sum_{k=1}^{y_0}( X_{t_k} - M^{y_0-k}_{t_{k+1}}
+ M^{y_0-k}_{t_k} ) \bigg|X_0=x\right].
\end{equation*}
The dual formulation in this form coincides with the result obtained in \cite{Sch}.}
\end{rem}
%*************************************************************

In general, we need to solve the deterministic control problem along the path in order to use this approach. If we have
a good approximation to the value function, then we can use the martingale arising from its Doob 
decomposition, as this will be an approximation to the optimal martingale. 

We note that, if we are given a set of approximations to the value function, we can bound the error made in the upper
bound arising from the dual formulation in terms of what are essentially the errors in the dynamic programming equations.

More specifically, let $V_t(y,x)$, $t=0,\dots,T$ be a set of approximations to the value function and denote by $\Vup_t(y,x)$ the 
associated upper bound on the value function.

%*************************************************************
\begin{lem}\label{lem_UpperLowerDifference}
The difference between the a priori estimate for the value function and the associated estimate arising from the dual
formulation can be expressed as
\begin{align*}
& \Vup_t(y,x)-V_t(y,x)= \\
& \qquad\qquad \be\left[\sup_{\pi\in\cp_{K,y,t}} \sum_{s=t}^{T-1} H_s(h_s,X_s)+\be_s\left[V_{s+1}(y_{s+1},X_{s+1})\right]-V_s(y_s,X_s)\bigg\vert X_t=x\right].
\end{align*}
\end{lem}
%*************************************************************

%*************************************************************
\begin{proof}
We choose the martingale in the upper bound to be the one generated by $V_t(y,x)$, and, thus,
\[ 
M^{y}_{t+1}-M^{y}_t = V_{t+1}(y,X_{t+1}) - \be_t\left[ V_{t+1}(y,X_{t+1})\right].
\]
Substituting this into the dual formulation given in Theorem~\ref{_thm_dual_formulation}, we have
{
\allowdisplaybreaks
\begin{align*}
\Vup_t(y,x)&=
\be\left[\sup_{\pi\in\cp_{K,y,t}} \sum_{s=t}^{T-1}\left\{H_s(h_s,X_s)-M^{y_{s+1}}_{s+1}+M^{y_{s+1}}_s\right\} + H_T(h_T,X_T)\bigg\vert X_t=x\right] 
\\
&=
\be\bigg[\sup_{\pi\in\cp_{K,y,t}} \sum_{s=t}^{T-1}\left\{H_s(h_s,X_s)-V_{s+1}(y_{s+1},X_{s+1}) + \be_s\left[ V_{s+1}(y_{s+1},X_{s+1})\right]\right\} \\
& \qquad\qquad\qquad+ H_T(h_T,X_T)\bigg\vert X_t=x\bigg] 
\\
&=
\be\bigg[\sup_{\pi\in\cp_{K,y,t}} \sum_{s=t}^{T-1}\big\{H_s(h_s,X_s)-V_{s+1}(y_{s+1},X_{s+1})
+ V_s(y_s,X_s)\\
& \qquad\qquad\qquad\qquad + \be_s\left[ V_{s+1}(y_{s+1},X_{s+1})\right]-V_s(y_s,X_s)\big\}
+ H_T(h_T,X_T)\bigg\vert X_t=x\bigg] 
\\
&=
\be\bigg[\sup_{\pi\in\cp_{K,y,t}} \sum_{s=t}^{T-1}\left\{H_s(h_s,X_s) + \be_s\left[ V_{s+1}(y_{s+1},X_{s+1})\right]-V_s(y_s,X_s)\right\} \\ 
&\qquad\qquad\qquad +V_t(y_t,X_t)-V_T(y_T,X_T)+ H_T(h_T,X_T)\bigg\vert X_t=x\bigg] 
\\
&= V_t(y,x) + \be\left[\sup_{\pi\in\cp_{K,y,t}} \sum_{s=t}^{T-1} 
\left\{ H_s(h_s,X_s)+\be_s \left[V_{s+1}(y_{s+1},X_{s+1})\right]-V_s(y_s,X_s)\right\}\bigg\vert X_t=x\right] 
\end{align*}
}
as $V_T(y_T,X_T) = H_T(h_T,X_T)$, giving the required result.
\end{proof}
%*************************************************************

%*************************************************************
%*************************************************************
\section{The numerical approach}
%*************************************************************
%*************************************************************

% why solve it? - practical problems described as stochastic control problem
% why solve it using probabilistic algorithm? - dimensionality
% 	relation to existing concepts (MC approach to swing, gas storage papers; PDE approach; 
% 	dimensionality, disadvantages, how they compare to this)
% what's point in dual formulation? - high biased estimates
% brief overview of the algorithm
% justification - only empirically; performance evaluation, when do we expect it to work? (compact constraint sets, concave payoff fns, something on the law of the process - elliptic??) 

We now present a numerical implementation of the dual upper bound derived in 
Lemma \ref{lem_UpperLowerDifference}. 
The lemma gives a representation of the difference between an upper bound $\Vup_0(y,x)$ and another 
(a priori) approximation $V_t(y,x)$ of $V^{K,*}_t(y,x)$.  In Section \ref{_seq_num_upper_bound}, we 
present a numerical method that approximates $\Vup_0(y,x)$ given that approximations of the functions 
$V_t(y,x)$ and
\begin{equation}
(x,y)\mapsto \be\left[V_{t+1}(y,X_{t+1})|X_t=x\right]
\label{_eq_conditional_expectation_of_candidate_solution}
\end{equation}
are available. 

In Section \ref{_seq_num_lower_bound}, we introduce an approach to generate an a priori estimate $V_t(y,x)$ and an approximation of the conditional expectation \eqref{_eq_conditional_expectation_of_candidate_solution}.  
%*************************************************************
\subsection{Estimating the dual upper bound}\label{_seq_num_upper_bound}
%*************************************************************
In this section, we assume that a set of a priori approximations $V_t(y,x)$ is available, i.e., for $t=0,\dots,T-1$, the function $V_t(y,x)$ represents an approximation of $V^{K,*}_t(y,x)$.
Furthermore, we assume that, for $t=0,\dots,T-1$, the function $V_t(x,y)$ and (an estimate of) 
\[
(x,y)\mapsto \be\left[V_{t+1}(y,X_{t+1})|X_t=x\right]
\]
can be computed for any time-$t$ reachable pair $(x,y)$.% in the domain of $V^K_t$. 

Under such assumptions, we introduce a numerical method that implements the upper estimate $\Vup_0(y,x)$ derived in Lemma \ref{lem_UpperLowerDifference}. 
Lemma \ref{lem_UpperLowerDifference} requires the estimation of a path-wise optimum. Hence, given a trajectory $x_{\cdot}=\{x_0,\dots,x_T\}$, we aim to approximate the function
\[
F_t(y,x_{\cdot}) := \sup_{\pi\in\cp_{K,y,t}} \sum_{s=t}^{T-1} \big\{H_s(h_s,x_s)+\be\left[V_{s+1}(y_{s+1},X_{s+1})|X_s=x_s\right]-V_s(y_s,x_s)\big\}
\]
recursively for $t=T,T-1,\dots,0$. The optimization algorithm is based on the following path-wise dynamic programming principle.
\begin{align} 
F_T(y,x_{\cdot})&=0,\nonumber\\ 
F_t(y,x_{\cdot})&=\sup_{\pi\in\cp_{K,y,t}}\left\{ \sum_{s=t}^{T-1} \big\{H_s(h_s,x_s)+\be\left[V_{s+1}(y_{s+1},X_{s+1})|X_s=x_s\right]-V_s(y_s,x_s)\big\}\right\}\nonumber\\
&=\sup_{h\in K_t(y,x_t)}\Big\{ H_t(h,x_t)+\be\left[V_{t+1}(y-h,X_{t+1})|X_t=x_t\right]\nonumber\\
& \ \ \ \ \ \ \ \ \ \ \ \ \ \ \ \ \ \ -V_t(y,x_t)+F_{t+1}(y-h,x_{\cdot})\Big\},\label{eq_PWDynamicProgrammingPrinciple}
\end{align}
and 
\begin{equation}
\Vup_0(y,x) =
V_0(y,x) + \be\left[F_0(y,X_{\cdot})|X_0=x\right].
\label{eq_DualRepresentationWithDynamicProgramming}
\end{equation}
Based on \eqref{eq_PWDynamicProgrammingPrinciple} and \eqref{eq_DualRepresentationWithDynamicProgramming}, we are now in a position to formulate the following algorithm.

% %**********************************************
% \paragraph*{Outline of the algorithm}
% %**********************************************

\begin{alg}\label{Main_Alg}
Generate $N$ independent trajectories $x^i_{\cdot}$, $i=1,\dots,N$ of the process $X$ started at a fixed $X_0$. For $i=1,\dots,N$
\begin{enumerate}
\item Set $t=T$, and define $y\mapsto\widehat{F}_T(y,x^i_{\cdot})=0$. 
\item\label{_step2} Set $t-1\to t$.
\item\label{_step3} Define a finite gird $\cg^y_t\subseteq \text{Dom}(\widehat{F}_{t}(\cdot,x^i_{\cdot}))\subseteq\br^l$ (see Remark \ref{_rem_dual_version}), and for each $y\in\cg^y_t$ solve the optimization problem
\begin{align*} 
\overline{F}_t(y,x^i_{\cdot})&=
\sup_{
\begin{smallmatrix}
y-h\in\text{Dom}(\widehat{F}_{t+1}(\cdot,x^i_{\cdot}))\\
h\in K_t(y,x^i_t)
\end{smallmatrix}
}
\Big\{ H_t(h,x^i_t)+\be\left[V_{t+1}(y-h,X_{t+1})|X_t=x^i_t\right]\\
& \ \ \ \ \ \ \ \ \ \ \ \ \ \ \ \ \ \ \ \ \ \ \ \ -V_t(y,x^i_t)+\widehat{F}_{t+1}(y-h,x^i_{\cdot})\Big\}
\end{align*}
\item\label{_step4} Given the set $\{ (y,\overline{F}_t(y,x^i_{\cdot}))| y\in \cg^y_t\}$, define (interpolate) $\widehat{F}_t(\cdot,x^i_{\cdot})$ on the whole domain $\text{Dom}(\widehat{F}_{t}(\cdot,x^i_{\cdot}))$ (see Remark \ref{_rem_dual_version}).  %(For particular techniques, see sections \ref{_sec_fy1} and \ref{_sec_fy2}.)
\item If $t\ge 1$, continue with \ref{_step2}, otherwise finish. 
\end{enumerate}
Once $y\mapsto \widehat{F}_0(y,x^i_{\cdot})$ is defined for all $i=1,\dots,N$, we approximate $\Vup_0(y,X_0)$ by the Monte-Carlo average
\begin{equation*}
V_0(y,X_0) + \tfrac{1}{N}\sum_{i=1}^N \widehat{F}_0(y,x^i_{\cdot}).
\end{equation*}
\end{alg}

Clearly, the main challenge in the implementation of Algorithm \ref{Main_Alg} is the solution of the optimization problem in step 3.

\begin{rem}\label{_rem_dual_version}{\rm 
The particular implementations of the above algorithm differ in 
\begin{enumerate}[(i)]
\item the specification of the domain $\text{Dom}(\widehat{F}_{t}(\cdot,x^i_{\cdot}))$,
\item the definition of $\cg^y_t$, 
\item the approximation of the solution to the optimization problem in point \ref{_step3} of the algorithm,
\item and the method applied in point \ref{_step4} of the algorithm.
\end{enumerate}
} 
\end{rem}

Two possible versions of Algorithm \ref{Main_Alg} are presented in Sections \ref{_sec_fy1} and \ref{_sec_fy2}.

%*************************************************************
\subsubsection{Implementation I: Discretization of the control}\label{_sec_fy1}
%*************************************************************
One possible approach is to discretized the problem in the control. We define $\cg^y_0$ as an (equidistant) 
grid contained in the set of initial control values of interest. Then, recursively for $t=0,\dots,T-1$, we define 
$\cg^y_{t+1}$ to be an (equidistant) grid contained in the set
\[
\{
y+h | y\in\cg^y_{t}, \ h\in K_t(y,x_t)
\}.
\]
Furthermore, $\text{Dom}(\widehat{F}_{t}(\cdot,x^i_{\cdot}))$ is defined to be the same as $\cg^y_t$;
this specification implies that the optimization problem in step \ref{_step3} of Algorithm \ref{Main_Alg} is an 
optimization over a finite set; moreover, $\overline{F}_t(\cdot,\cdot)=\widehat{F}_t(\cdot,\cdot)$ for $t=0,\dots,T$. 

\begin{rem}\label{rem_YGrid}{\rm 
The choice of $\cg^y_t$ depends on the constraints of the problem. For instance, in the case of the gas 
storage problem, there is a well defined lower and upper limit of $y$; $\cg^y_t$ can be an equidistant grid 
in this region. }
\end{rem}

%*************************************************************
\subsubsection{Implementation II: Parametric curve fitting}\label{_sec_fy2}
%*************************************************************
Here, we define $\cg^y_t$ similarly to the previous version. However, we assume $\widehat{F}_t(\cdot,\cdot)$ to be a parametric surface of the following form.
\[
\widehat{F}_t(y,x^i_{\cdot})=\sum_{r=1}^R \lambda^i_{t,r}\phi_r(y)
\]
for some vector of parameters $\Lambda^i_t=(\lambda^i_{t,1},\dots,\lambda^i_{t,R})$ depending on the 
trajectory $x^i_{\cdot}$ and for some set of test functions $(\phi_1,\dots,\phi_R)$ with domains in $\br^l$, implying
\[
\text{Dom}(\widehat{F}_{t}(\cdot,x^i_{\cdot}))
=
\bigcap_{r=1}^R \text{Dom}(\phi_r).
\]
The accuracy of the algorithm is sensitive to the choice of test functions; more specifically, different settings 
may have different optimal sets of test functions, and the (numerical) solution of the optimization problem 
in step \ref{_step3} of Algorithm \ref{Main_Alg} should be adapted to the particular choice of test functions. 

Point \ref{_step4} of Algorithm \ref{Main_Alg} is implemented via a least squares regression, i.e., we define $\Lambda^i_t$  
to minimize the expression
\[
\sum_{y\in\cg^y_t} \left[ \overline{F}_t(y,x^i_{\cdot}) - \sum_{r=1}^R \lambda^i_{t,r}\phi_r(y) \right]^2.
\]

\begin{rem}{\rm 
As we increase the number $R$ of independent test functions and the number $N$ of simulated trajectories,
we anticipate that $\widehat{F}_t(\cdot,\cdot)$ converges to $F_t(\cdot,\cdots)$ for $t=0,\dots,T$. 
However, $\widehat{F}_t(\cdot,\cdot)$ is likely to estimate $F_t(\cdot,\cdots)$ from below, and, therefore, our 
method is likely to result in a low-biased estimate of the dual formulation based upper bound.}
%
%We note that we are assuming exact optimisation in point \ref{_step3}, in that as both the number $R$ of 
%independent test functions and the number $N$ of independent trajectories of the process $X$ tend to 
%infinity, $\widehat{F}_t(\cdot,\cdot)$ converges to $F_t(\cdot,\cdot)$ for $t=0,1,\dots,T$. 
%%REF!!!!!!
%However, with numerical optimisation in point \ref{_step3}, this method might approximate the dual 
%upper bound from below. }
\end{rem}

%*************************************************************
\subsection{An a priori estimate}\label{_seq_num_lower_bound}
%*************************************************************
%The a priori estimate is based on the dynamic programming formulation of the problem. 

As stated at the beginning of Section \ref{_seq_num_upper_bound}, the solution of \eqref{eq_PWDynamicProgrammingPrinciple} requires computable functions
%\begin{align}
%(y,x) & \mapsto V_t(y,x),\\ 
%(y,x) & \mapsto \be\left[V_{t+1}(y,X_{t+1})|X_t=x\right]\label{eq_ConditionalExpectation}
%\end{align}
$V_t(y,x)$ and 
\[
G_t(y,x):=\be\left[V_{t+1}(y,X_{t+1})|X_t=x\right]
\]
approximating $V^{K,*}_t(y,x)$ and $\be\left[V^{K,*}_{t+1}(y,X_{t+1})|X_t=x\right]$, respectively, for $t=0,\dots,T-1$. 
We suggest the following method, which is based on the dynamic programming formulation.

%*************************************************************
\begin{defn}(Dynamic Programming Formulation)
\begin{align}
V^{K,*}_t(y,x) := \left\{
\begin{array}{ll}
\sup_{h\in K_T(y,x)} H_T(h,x) &  \text{ if } t=T,\\
\sup_{h\in K_t(y,x)} \left\{H_t(h,x) + \be\left[V^{K,*}_{t+1}(y-h,X_{t+1})|X_t=x\right]\right\} & \text{ if } 0\le t\le T-1.
\end{array}
\right.
\label{eq_MainDynamicProgrammingFromulation}
\end{align}
\end{defn}
%*************************************************************
%In lemma lemma \ref{lem_MainDynamicProgrammingFromulation} (appendix \ref{_sec_DynamicProgrammingFormulation}), we prove that the dynamic programming formulation holds almost surely for time-$t$ reachable pairs $(y,x)$, where we call a pair $(x,y)$ time-$t$ reachable, if  
%there exist $y_0\in Y_0^K$, a policy $\pi\in \mathcal{P}_{K,y_0,0}$ and $\omega\in\Omega$, such that when $\pi$ is started at $y_0$, then the equalities $y_t(\omega):=y_0-\sum_{s=0}^{t-1}h_s(\omega)=y$ and $X_t(\omega)=x$ hold. 

For the computation of the conditional expectation in the above formulation, we introduce a slightly extended version of the standard least squares regression based Monte Carlo method \cite{LonSchw,VanRoyTsits,ClemLamPro}. 
Our construction yields an a priori estimate $V_0(\cdot,\cdot)$ that approximates $V^{K,*}_0(\cdot,\cdot)$ for a bounded %and measurable
set $S$ of initial $X_0$ values, where $S$ is contained in the support of the law of $X_0$. 

%\paragraph*{Outline of the algorithm}
\begin{alg}\label{Main_Alg_2}
Define a set $\cg^x_0$ of distinct initial values of $X$ in $S$ (see Section \ref{_sec_practical_implementation_XYregression}) and generate independent trajectories $x^i_{\cdot}$, $i=1,\dots,N$, of the process $X$, with $x^i_0=x$ for each $x\in\cg^x_0$. 
Define $\cg^x_t= \{ x^i_t \ | \ 1 \le i \le N\}$ for $t=1,\dots,T$, where $N=|\cg^x_0|$. Furthermore, for $t=1,\dots, T$, define a finite set 
\[
\cg^{yx}_t\subseteq \text{Dom}(V_t(\cdot,\cdot))\subseteq \br^l\times \br^d 
\]
such that, for all $(y,x)\in\cg^{yx}_t$, we have $x\in\cg^x_t$ (see Remark \ref{_rem_versions_of_XYregression}). 
Then, proceed as follows.
\begin{enumerate}
\item Set $t=T$ and define 
\[
V_T(y,x) = \sup_{h\in K_T(y,x)} H_T(h,x). 
\]
\item\label{_apriori_step2} Set $t-1\to t$.
\item\label{_apriori_step3} Given the set 
\[
\left\{\big(y,x_t,x_{t+1},V_{t+1}(y,x_{t+1})\big)\ | \ (y,x_{t+1})\in\cg^{yx}_{t+1}\right\},
\]
define a function $\widehat{G}_t(y,x)$ approximating $G_t(y,x)$ on $\text{Dom}(\widehat{G}_t(\cdot,\cdot))$ (see Remark \ref{_rem_versions_of_XYregression}). 
%\item Define a finite grid $\cg^{yx}_t\in \br^l\times \br^d$, such that for all $(y,x)\in\cg^{yx}_t$, we have $x\in\cg^x_t$ (see Remark/Section XXX).  
\item\label{_apriori_step4} For each $(y,x)\in\cg^{yx}_t$, solve the optimization problem
\begin{equation}
\overline{V}_t(y,x) = \sup_{
\begin{smallmatrix}
y-\hat{y} \in K_t(y,x) \\
(\hat{y},x) \in \text{Dom}(\widehat{G}_t(\cdot,\cdot))
\end{smallmatrix}
}
\left\{
H_t(y-\hat{y},x)+ \widehat{G}_t(\hat{y},x)
\right\}.
\label{_eq_apriori_optimisation}
\end{equation}
\item\label{_apriori_step5} $\overline{V}_t(\cdot,\cdot)$ is only defined on $\cg^{yx}_t$. Given the set
\[
\left\{\big(y,x,\overline{V}_{t}(y,x)\big)\ | \ (y,x)\in\cg^{yx}_t\right\},
\]
define the function $V_t(\cdot,\cdot)$ on its domain (see Remark \ref{_rem_versions_of_XYregression}). 
\item If $t\ge 1$, then continue with step $2$; else, $V_0(y,x)$ results in an a priori approximation. 
\end{enumerate}
\end{alg}

The above outline of Algorithm \ref{Main_Alg_2} leaves some choice as to how certain things are done in detail; in particular, this includes the following points.

\begin{rem}\label{_rem_versions_of_XYregression}{\rm
The particular implementations of Algorithm \ref{Main_Alg_2} differ in 
\begin{enumerate}[(i)]
\item the construction of the set $\cg^{yx}_t$, 
\item the construction of the function $\widehat{G}_t(\cdot,\cdot)$ in step \ref{_apriori_step4}, 
\item and the construction of the function $V_t(\cdot,\cdot)$ in step \ref{_apriori_step5}. 
\end{enumerate}
In Section \ref{_sec_practical_implementation_XYregression}, we implement a particular version of Algorithm \ref{Main_Alg_2} for the a priori estimate. }
\end{rem}

%*************************************************************
\subsubsection{Choosing an implementation}\label{_sec_practical_implementation_XYregression}
%*************************************************************
Since the three items described in Remark \ref{_rem_versions_of_XYregression} are closely connected, we discuss them together. 

%For the numerical solution of the optimisation problem \eqref{eq_MainDynamicProgrammingFromulation}, an accurate approximation of the conditional expectation \eqref{eq_ConditionalExpectation} is required. Here, we consider the orthogonal projection of the function \eqref{eq_ConditionalExpectation} onto a function space spanned by certain test functions $(\psi_1,\dots,\psi_Q)$, that is we consider approximations of the following form

In a similar way to the classic least squares regression based approach (cf.\,\cite{LonSchw,VanRoyTsits,ClemLamPro}), we 
approximate the function $G_t(\cdot,\cdot)$ by an orthogonal projection onto a function space spanned by a 
set of test functions 
$\{\psi_1,\dots,\psi_Q\}$, where, for $q=1,\dots,Q$, $\psi_q$ is defined on $\text{Dom}(G_t(\cdot,\cdot))\subseteq \br^l\times\br^d$, and
\begin{equation}
\widehat{G}_t(y,x)=\sum_{q=1}^Q\gamma_{t,q}\psi_q(y,x)\approx G_t(y,x)=\be\left[V_{t+1}(y,X_{t+1})|X_t=x\right].
\label{_eq_xyRegression}
\end{equation}
In contrast to \cite{LonSchw,VanRoyTsits,ClemLamPro}, where the orthogonal projection at time $t$ is determined by the distribution of $X_t$, we have to deal with the control variable as well. 
We define the projection to minimize
\begin{equation}
\be_{Y,Z}\left[\left(\be[V_{t+1}(Y,X_{t+1})|X_t=Z]-\sum_{q=1}^Q\gamma_{t,q}\psi_q(Y,Z)\right)^2\right],
\label{_eq_XYreg_error}
\end{equation}
where $Z$ and $Y$ are independent random variables. In most applications, the set 
reachable by $y_{t+1}$ is bounded,
%support of $y_{t+1}$ is bounded, 
and, therefore, in order to try to obtain uniform accuracy across the reachable set, we will take $Y$ to be uniformly 
distributed on this bounded set.  The distribution of $Z$ can be defined to coincide with the distribution of $X_t$.
However, in many applications, such as the numerical example in Section \ref{_sec_numerical_results}, only the
distribution of $X_t$ conditioned on particular values of $X_0$ is specified; here, we assume 
that the law of $X_0$ is uniform on a certain set $S$. 

%the conditional law of $X_t$ and the (bounded and measurable) set $S$ of initial $X_0$ values. We assume $X_0$ to be uniformly distributed on $S$ and
%\[
%\mathbb{P}_Z[Z\in A] = \int_S \mathbb{P}[X_t \in A| X_0=x] \mu_S(\text{d} x),
%\] 
%for any Borel set $A\subseteq \br^d$, where $\mu_S$ denotes the uniform measure on $S$. 

Formula \eqref{_eq_XYreg_error} suggests that, by increasing the number of appropriately chosen test functions $\psi_1,\psi_2,\dots$, $\widehat{G}_t(\cdot,\cdot)$ approximates the conditional expectation
\begin{equation*}
(y,x)\mapsto \be[V_{t+1}(y,X_{t+1})|X_t=x]
\end{equation*}
in the mean square sense with respect to the joint measure of $Y$ and $Z$; for our particular choice of test functions, see Section \ref{_sec_results_APrioriEstimate}.

In order to determine the regression coefficients $\gamma_{t,q}$ for $q=1,\dots,Q$, we observe that, when \eqref{_eq_XYreg_error} is minimized, we have
\begin{align*}
& \frac{\partial}{\partial \gamma_{t,r}} \be_{Y,Z}\left[\left(\be[V_{t+1}(Y,X_{t+1})|X_t=Z]-\sum_{q=1}^Q\gamma_{t,q}\psi_q(Y,Z)\right)^2\right]  \\
& \ \ \ \ = 2\be_{Y,Z}\big[\be[V_{t+1}(Y,X_{t+1})|X_t=Z]\psi_r(Y,Z)\big] 
		 -2\be_{Y,Z}\left[\sum_{q=1}^Q\gamma_{t,q}\psi_q(Y,Z)\psi_r(Y,Z)\right]\\
& \ \ \ \ = 2\be_{Y,Z}\big[V_{t+1}(Y,X_{t+1})\psi_r(Y,Z)\big] 
		 -2\sum_{q=1}^Q\gamma_{t,q}\be_{Y,Z}\left[\psi_q(Y,Z)\psi_r(Y,Z)\right]=0.
\end{align*}
Hence, $\gamma_t=(\gamma_{t,1},\dots,\gamma_{t,Q})^T$ satisfies the linear equation
\begin{equation}
B_{V,\psi}=B_{\psi}\gamma_t,
\label{_eq_gamma_equation}
\end{equation}
where
\[
B_{V,\psi}=\be_{Y,Z}\big[V_{t+1}(Y,X_{t+1})\psi(Y,Z)\big]
\]
and
\[
B_{\psi}=\be_{Y,Z}\big[\psi(Y,Z)\psi(Y,Z)^T\big]
\]
for $\psi(x,y)=(\psi_1(x,y),\dots,\psi_Q(x,y))^T$.

When estimating the regression coefficients, we replace $B_{V,\psi}$ and $B_{\psi}$ in \eqref{_eq_gamma_equation} with their Monte-Carlo estimates
\begin{align}
\widehat{B}_{V,\psi}&:=\tfrac{1}{|\cg^{yx}_{t+1}|}\sum_{(y,x_{t+1})\in\cg^{yx}_{t+1}}V_{t+1}(y,x_{t+1})\psi(y,x_t)\label{_eq_BV}\\
\text{and}\quad\widehat{B}_{\psi}&:=\tfrac{1}{|\cg^{yx}_{t+1}|}\sum_{(y,x_{t+1})\in\cg^{yx}_{t+1}}\psi(y,x_t)\psi(y,x_t)^T.\label{_eq_Bpsi}
\end{align}

The choice of $\cg^x_0$ and $\cg^{yx}_t$, $t=0,\dots,T$, determines how accurately $\widehat{B}_{V,\psi}$ and $\widehat{B}_{\psi}$ approximate $B_{V,\psi}$ and $B_{\psi}$, respectively. When implementing the method, we consider
\begin{enumerate}[i)]
\item $\cg^x_0$ to be randomly sampled from the law of $X_0$, or $\cg^x_0$ to be a low discrepancy sequence in $S$,
\item $x_t$ to be randomly sampled from the conditional distribution $X_t|X_0$,
\item and $y$ to be independent of $x_t$ and randomly sampled from the uniform distribution on the support of $y_t$, or to be a low discrepancy sequence\footnote{We tested rank-$1$ lattices, see Section \ref{_sec_numerical_results}.} in the support of $y_t$; we generated a small number ($1$ to $10$) of $y$ items for each $x$. 
\end{enumerate}

\begin{rem}\label{_rem_xyGrid}{\rm
Initially, we looked at defining $\cg^{yx}_t=\cg^y_t\times\cg^x_t$, for some set $\cg^y_t$. However, the numerical results showed that to achieve a given accuracy, a large enough 
$\cg^y_t$ is required, resulting in a set $\cg^y_t\times\cg^x_t$ significantly larger than the size of $\cg^{yx}_t$ constructed in the version described prior to this remark (calibrated to yield the same accuracy). 

Figure \ref{fig:xygirds} demonstrates the difference between $\cg^y_t\times\cg^x_t$ and the set $\cg^{yx}_t$ described before this remark. 
We observe that $\cg^{yx}_t$ yields a better coverage with fewer grid points. }
\end{rem}

\begin{figure}[ht]
\centering
\subfigure[$\cg^{yx}_t$ grid, y points generated rank-1 lattice rule ($4000$ points in total, $8$ $y$-items per each $x$ item)]{
\includegraphics[trim = 32mm 72mm 28mm 83mm, clip, width=0.47\textwidth]{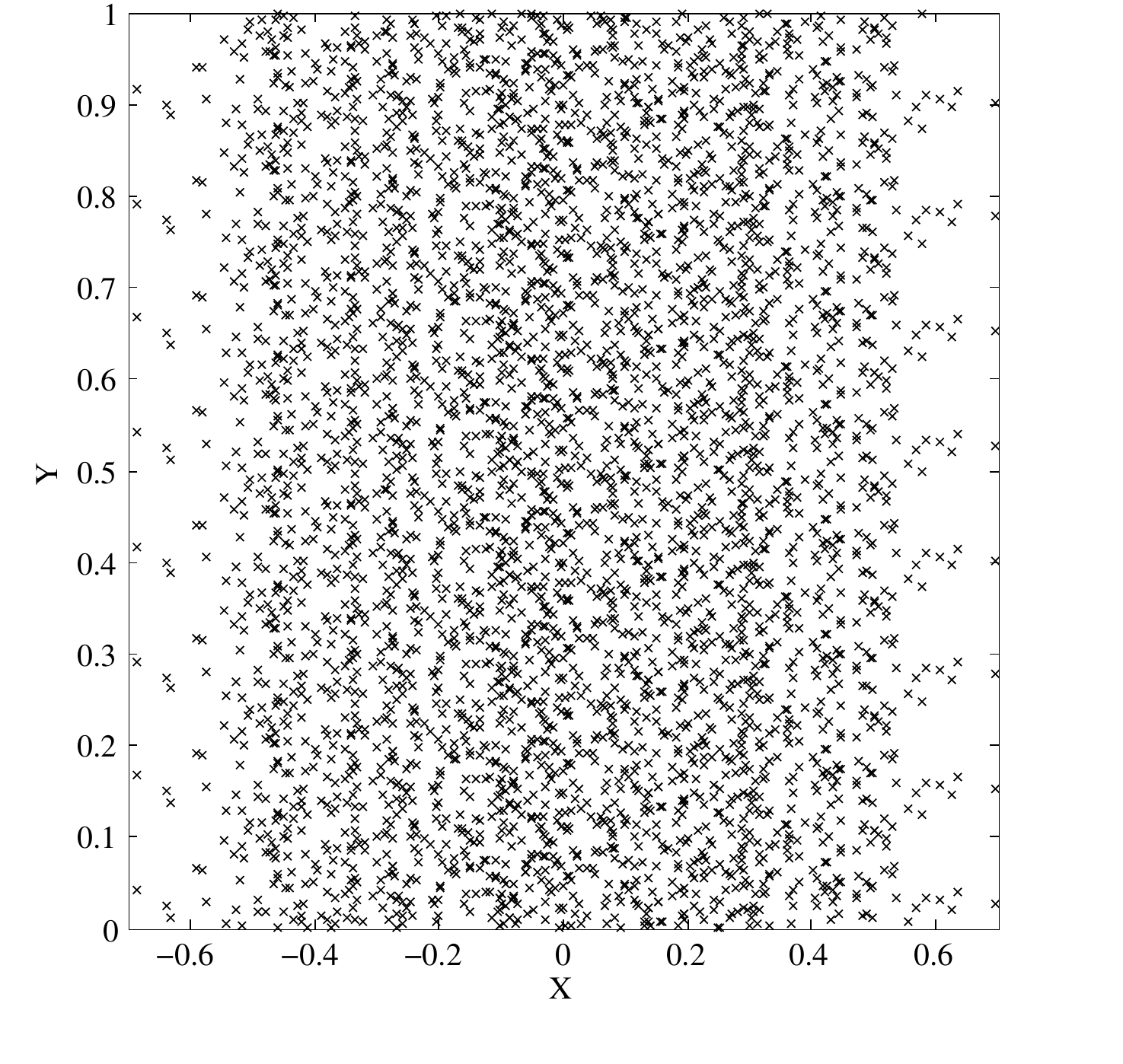}
% left bottom right top
\label{fig:xygrid1}
}
\subfigure[$\cg^{yx}_t=\cg^y_t\times\cg^x_t$ ($12500$ points in total, $25$ $y$-items per each $x$ item)]{
\includegraphics[trim = 32mm 72mm 28mm 83mm, clip, width=0.47\textwidth]{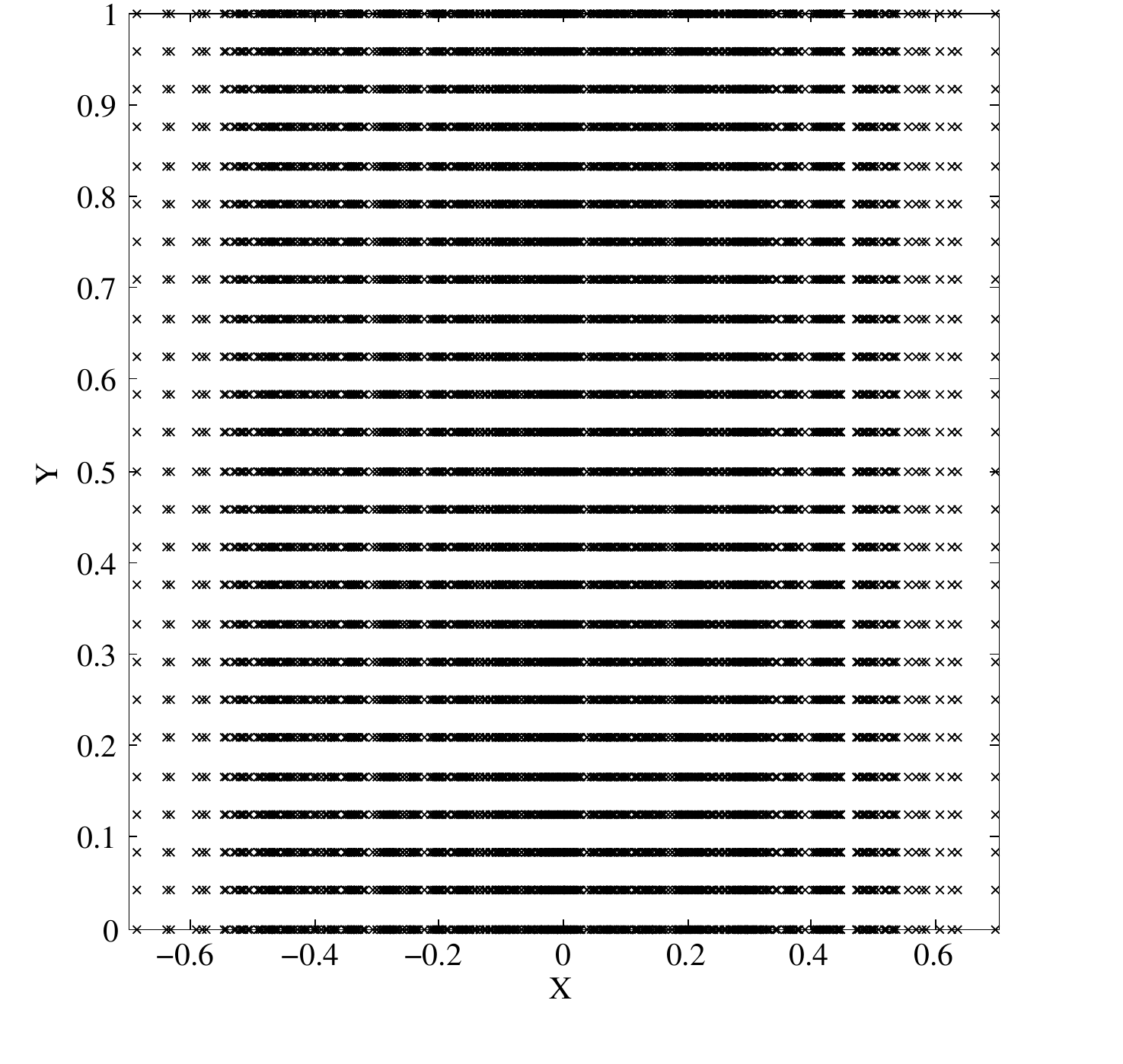}
% left bottom right top
\label{fig:xygrid2}
}
\caption{Different constructions of $\cg^{yx}_t$ based on the same $\cg^x_t$, assuming equidistant $\mathcal{G}^x_0\subset[-0.5,0.5]$ grid and Gaussian conditional distribution $(X_t|X_0)$.}
\label{fig:xygirds}
\end{figure}

What remains is to be specified are the particulars of step \ref{_apriori_step5} of Algorithm \ref{Main_Alg_2}, i.e., to define $V_t(\cdot,\cdot)$ given
\[
\left\{\big(y,x,\overline{V}_{t}(y,x)\big)\ | \ (y,x)\in\cg^{yx}_t\right\}.
\]
To do this, one can use interpolation, or one can fit a parametric surface to the graph of $\overline{V}_{t}(\cdot,\cdot)$; we consider the parametric representation
\[
V_t(y,x) = 	\sum_{q=1}^Q\beta_{t,q}\psi_q(y,x),
\]
choosing $\beta_{t,q}$, $q=1,\dots,Q$, to minimize the mean square error
\begin{equation*}
\sum_{(y,x)\in\cg^{yx}_t}\left(\overline{V}_{t}(y,x)-\sum_{q=1}^Q\beta_{t,q}\psi_q(y,x)\right)^2,
%\label{_eq_XYreg_error}
\end{equation*}
i.e., we define $V_t(\cdot,\cdot)$ by another least squares regression.

\begin{rem}{\rm 
For the numerical computation of the dual formulation based approach, to ensure that we have an upper bound, it is 
essential that the estimate of 
\begin{equation}
V_t(y,x)-\mathbb{E}[V_{t+1}(y,X_{t+1})|X_t=x]
\label{_eq_martingaleIncrement}
\end{equation}
is a martingale increment, that is it has zero expectation.

As introduced in this section, the function
\begin{equation}
V_t(y,x)-\widehat{G}_t(y,x)
\label{_eq_biasedMartingaleIncrement}
\end{equation}
is a biased estimate of \eqref{_eq_martingaleIncrement}. When \eqref{_eq_biasedMartingaleIncrement} is used 
for computing the dual upper bound - in particular, when $\widehat{G}_t(\cdot,\cdot)$ is a poor estimate of 
$G_t(\cdot,\cdot)$ - the error due to this bias might be larger than the statistical error. 

In such cases, one can replace $\widehat{G}_t(\cdot,\cdot)$ with the following estimate
\[
\widetilde{G}_t(y,x)=\tfrac{1}{L}\sum_{i=1}^LV_{t+1}(y,X_{t+1}^{(i)})
\]
where $\{X_{t+1}^{(i)}$, $i=1,\dots,L\}$ for some $L$ is an i.i.d. sample from the conditional distribution of 
$(X_{t+1}|x)$. }
\end{rem}

%*************************************************************
\subsubsection{Other choices in the implementation}
%*************************************************************

Multivariate regression similar to \eqref{_eq_xyRegression} has been mentioned in 
\cite{BooJon}. However, \cite{BooJon} does not pursue the same route as presented above, but rather restricts
attention to a least squares regression that uses test functions depending only on the underlying factor $X$, and, for 
each value of $y$ in a finite set $\cg^y\subset\br^l$, a separate simpler regression is computed. 
The extension of $\overline{V}_t(\cdot,\cdot)$ to the whole domain of $V_t(\cdot,\cdot)$ is not considered 
(step \ref{_apriori_step5} in Algorithm \ref{Main_Alg_2}). 
\cite{BooJon} restricts the optimization problem \eqref{_eq_apriori_optimisation} in step \ref{_apriori_step5} to $\cg^y$. 

Computing regressions which are based on test functions depending only on $X$ is less expensive than a regression 
with high number of $(y,X)$-dependent test functions (see Section \ref{_sec_results_APrioriEstimate} for the 
implementation of the a priori method). However, for accurate estimates, a fine grid $\cg^y$ is required, and, hence, 
a high number of simple regressions needs to be computed.  
With carefully chosen $\mathcal{G}^{yx}_t$-grid (see Remark \ref{_rem_xyGrid}) and a suitable set of 
$(y,X)$-dependent test functions, our version attains the same accuracy at significantly lower cost.

%A large enough $\cg^y$ will yield an a priori estimate close to the one generated by the a priori estimate introduced in section \ref{_sec_practical_implementation_XYregression}, however in the numerical examples we considered, the latter method is less expensive. 

%*************************************************************
\subsubsection{A note on low biased methods}\label{_sec_lowBiased}
%*************************************************************
%The a priory estimate can be used to generate a low biased estimate of $V^K()$...
The a priori estimates $V_t(y,x)$ and $\mathbb{E}[V_{t+1}(y,X_{t+1})|X_t=x]$, described above, typically result in a high biased estimate of the value function. However, the outcome of the a priori method can be applied to generate a low biased estimate. 
By definition, for any $y_0\in Y_0^K$ and for any policy $\pi\in\mathcal{P}_{K,y_0,0}$, $V_0^{K,\pi}(y_0,\cdot)$ is a low biased estimate of $V^{K,*}_0(y_0,\cdot)$. 

The a priori estimate generates a policy as follows. For $t\in\{0,\dots,T-1\}$, a reachable pair $(y,x)$, and $\epsilon>0$, there exists at least one value $\hat{h}_t\in K_t(y,x)$ that satisfies
\begin{align}
& H_t(\hat{h}_t,x)+\mathbb{E}[V_{t+1}(y-\hat{h}_t,X_{t+1})|X_t=x]\nonumber\\
\ & \sup_{h\in K_t(y,x)}\left\{H_t(h,x)+\mathbb{E}[V_{t+1}(y-h,X_{t+1})|X_t=x]\right\}-\epsilon.
\label{_eq_lowBiasedPolicy}
\end{align}
Given a starting value $y_0\in Y_0^K$, and assuming that the supremum exists almost surely for reachable pairs,
\eqref{_eq_lowBiasedPolicy} determines a exercise policy $\hat{\pi}=(\hat{h}_0,\dots,\hat{h}_T)\in\mathcal{P}_{K,y_0,0}$.
As $\epsilon$ approaches $0$, $V^{K,\hat{\pi}}_0(\cdot,\cdot)$ converges to $V^{K,*}_0(\cdot,\cdot)$. This convergence result motivates the following low-biased algorithm.

% %**********************************************
% \paragraph*{Outline of the algorithm}
% %**********************************************

\begin{alg}\label{Main_Alg_3}
Fix $y_0\in Y_0^K$ and $\epsilon>0$. Generate $N$ independent trajectories $x^i_{\cdot}$, $i=1,\dots,N$, of the process $X$ started at a fixed $X_0$. For $i=1,\dots,N$,
\begin{enumerate}
\item set $t=0$, and $V^i=0$,
\item\label{_forwardStep} for $x=x_t^i$ and $y=y_t$, find a $\hat{h}_t$ that satisfies \eqref{_eq_lowBiasedPolicy},
\item set $V^i+H_t(\hat{h}_t,x^i_t) \to V^i$.
\item and, if $t=T$, then stop; else, set $t+1\to t$, and continue with step \ref{_forwardStep}.
\end{enumerate}

Once this routine has been executed for all $i=1,\dots,N$, the Monte-Carlo average 
\[
V_0^{\downarrow}(y_0,X_0):=\tfrac{1}{N}\sum_{i=1}^N V^i
\] 
approximates (up to statistical error due to sampling variance) a low biased estimate at time $0$ for initial control value $y_0$ and initial factor value $x$.
\end{alg}

%*************************************************************
%*************************************************************
\section{Numerical results}\label{_sec_numerical_results}
%*************************************************************
%*************************************************************
In this section, we discuss the gas storage example following \cite{ThoDavRas}, and we compare the 
numerical performance of the implementation of both, the a priori method and the method based on the dual formulation. 

The gas storage problem as well as related probabilistic numerical methods have also been discussed in \cite{LudCar} and in \cite{BooJon}.

\begin{figure}[p]
\centering
\includegraphics[trim = 18mm 98mm 18mm 98mm, clip, width=0.95\textwidth]{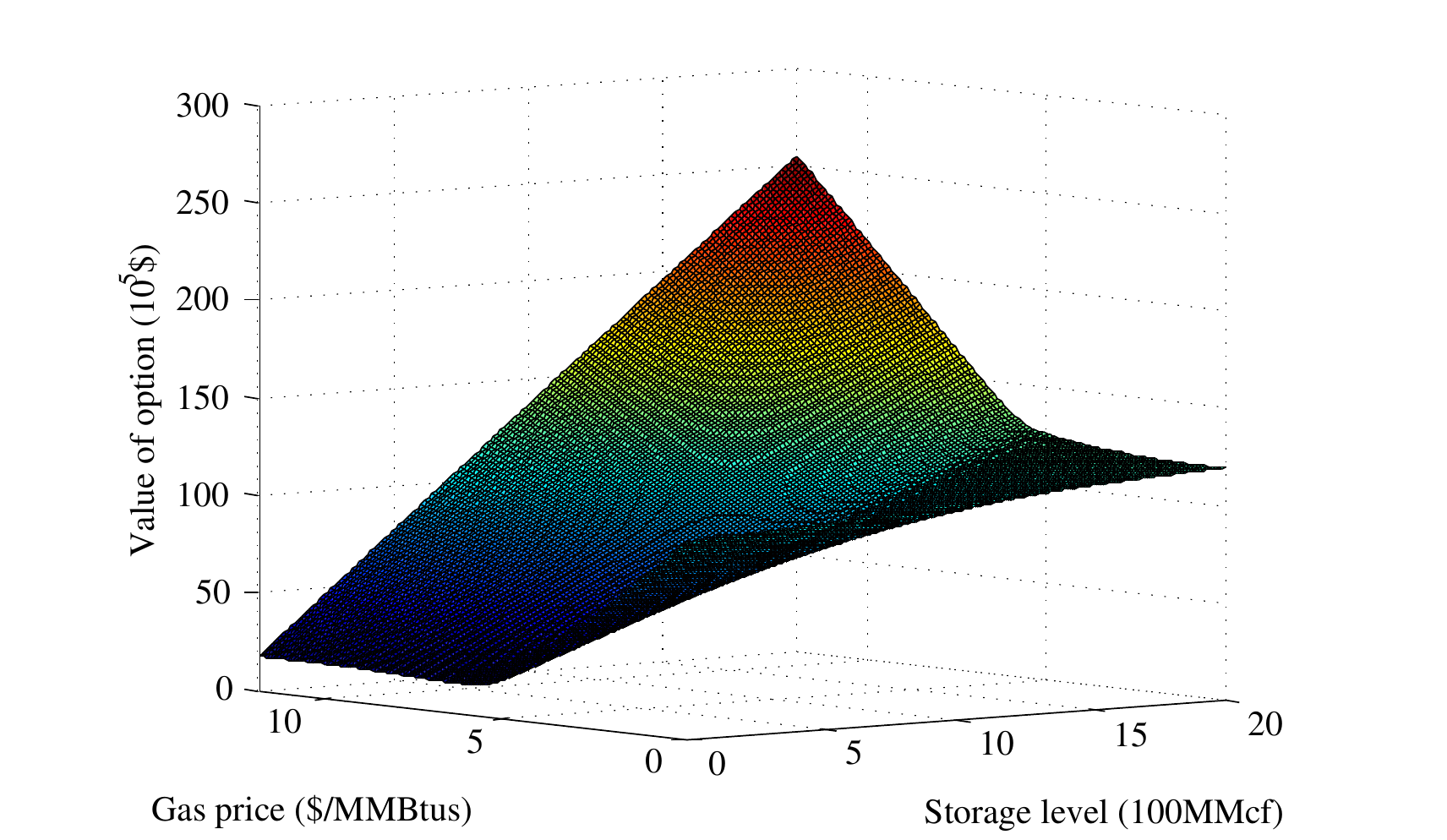}
% left bottom right top
\caption{A priori estimate of the value of option at time $0$.}
\label{fig:apriori1}
\end{figure}

\begin{figure}[p]
\centering
\includegraphics[trim = 18mm 98mm 18mm 98mm, clip, width=0.95\textwidth]{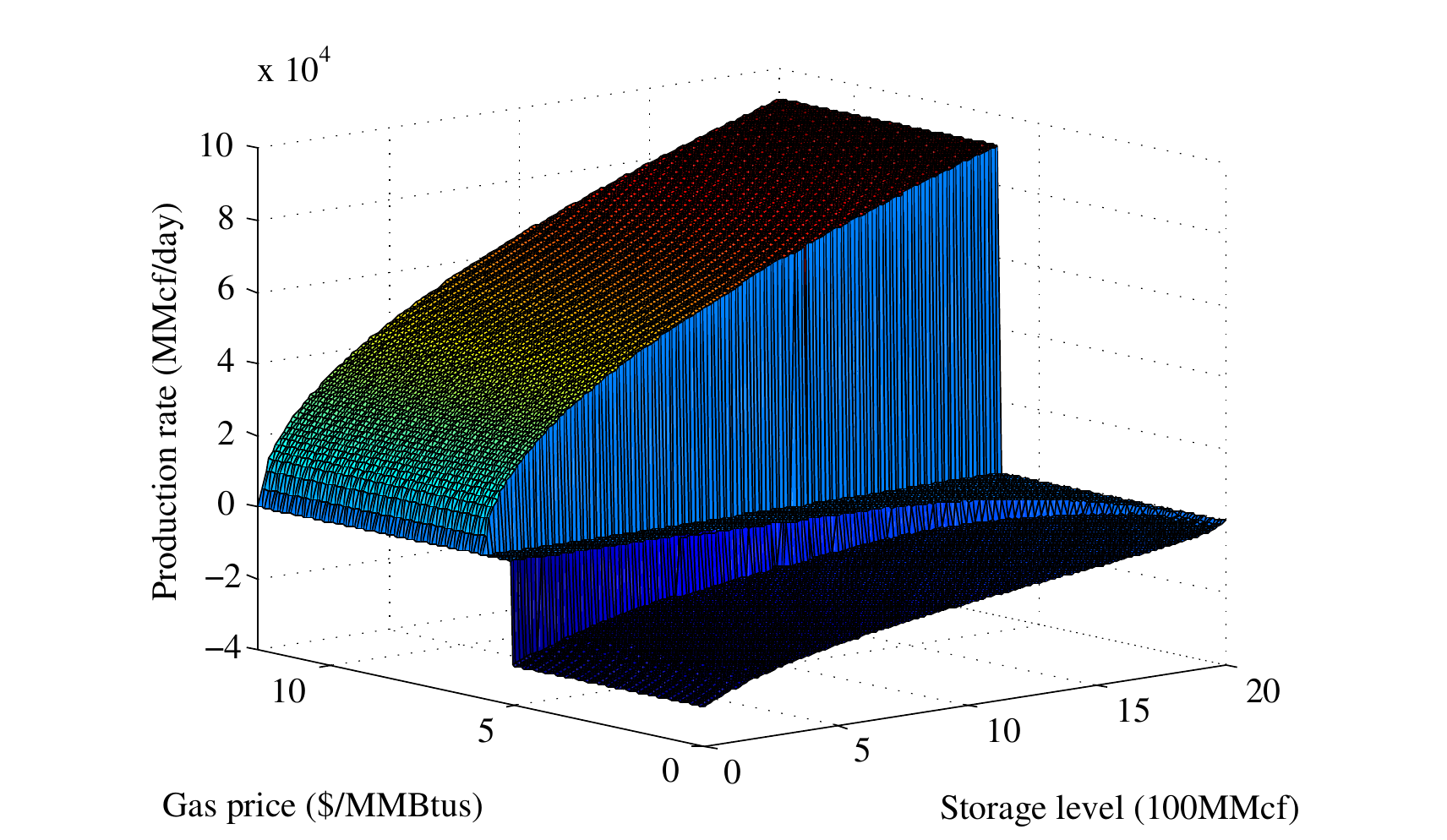}
% left bottom right top
\caption{A priori estimate of the optimal rate of production at time $0$.}
\label{fig:apriori2}
\end{figure}

%*************************************************************
\subsection{The gas storage problem}\label{_sec_gas_storage}
%*************************************************************

%%This could do with some more description

The natural gas storage problem addresses the optimal utilization of certain types of storage facilities. We assume relatively high deliverability and high injection rates. In particular, given the price $X_t$ of gas and the amount $y_t$ of \emph{working gas} in the inventory at time $t$, we aim to optimize the production (injection) amount for the given day, for each day over a year.  

We introduce the following notation. 
\begin{itemize}
\item $c$, the rate of production if $c>0$, or the rate of injection if $c<0$. The rate is measured in million cubic feet per day (MMcf$/$day). 
\item $y_{\text{base}}$, base gas requirement (built into the facility and cannot be removed).  
\item $y_{\max}$, the maximum storage capacity of the facility on top of the base gas level.
\item $c_{\max}(y)$, the maximum production rate at storage level $y$.
\item $c_{\min}(y)$, the maximum injection rate at storage level $y$.
\item $a(y,c)$, the rate of gas that is lost given production at rate $c>0$ or injection at rate $c<0$.
\item $r$, the discount rate. 
\end{itemize}

As in \cite{ThoDavRas}, we consider a facility with working gas capacity of $y_{\max}=2000$MMcf and with base gas requirement $y_{\text{base}}=500$MMcf. The maximum production rate (attainable at maximum capacity) is known to be $c_{\max}(y_{\max})=250$MMcf$/$day, whereas the maximum injection rate (attainable at minimum capacity) is $c_{\min}(0)=-80$MMcf$/$day. 
The facility is available for one year, and a decision on gas production/injection is made daily, i.e., $\mathcal{T}=\{0,1,\dots,365\}$. 

We assume that the loss rate satisfies
\[
a(y,c)=a(c)=\left\{
\begin{array}{cl}
0 & \text{ if } c\ge 0,\\
1.7 & \text{ if } c< 0.
\end{array}
\right.
\]

%In \cite{ThoDavRas}, the continuous time change in $y_t$ is described by an ordinary differential equation
%\[
%\text{d} y_t = -(c+a(y_t,c))\text{d} t.
%\]
In the discrete-time formulation\footnote{In \cite{ThoDavRas}, the continuous time production/injection is described by an ordinary differential equation. The discrete-time formulation is an approximation of the solution to that ODE.}, we approximate the daily delivered/injected amount by
\begin{equation}
h_t=y_t-y_{t+1}\approx c\Delta t,
\label{_eq_daily_amount}
\end{equation}
i.e., the unit of time is assumed to be a day (including weekend days), which means $\Delta t=1$. 

The daily constraints on gas production and injection are derived from the ideal gas law and Bernuolli's law (the reader is referred to Section 3 in \cite{ThoDavRas} for details\footnote{Note that, in this paper, the time unit is daily, whereas in \cite{ThoDavRas} the time is measured in years.}). In particular,
\begin{equation}
c_{\max}(y)=C_0 \sqrt{y},
\label{_eq_max_rate}
\end{equation}
where $C_0=c_{\max}(y_{\max})/\sqrt{y_{\max}}$. 
Moreover, 
\begin{equation}
c_{\min}(y)=-C_1\sqrt{\frac{1}{y+y_{\text{base}}}+C_2},
\label{_eq_min_rate}
\end{equation}
where $C_2=-1/(y_{\max}+y_{\text{base}})$ and 
\[
C_1 = c_{\min}(0) / \sqrt{\frac{1}{y_{\text{base}}}+C_2}.
\]
Combining \eqref{_eq_max_rate} and \eqref{_eq_min_rate} with \eqref{_eq_daily_amount}, we get the 
constraint set for the amount of gas that can be produced/injected during a day:
\begin{equation}
K_t(y,x)=K(y)=
[ -\min\{c_{\min}(y)\Delta t,y_{\max}-y\}, \min\{c_{\max}(y)\Delta T,y\} ].
\label{_eq_constraints}
\end{equation}

The payoff function is defined by $H_t(\cdot,\cdot)=0$ for $t=T$, and
\begin{equation}
H_t(h_t,X_t) = \left\{
\begin{array}{lc}
e^{-rt}h_tX_t & \text{ if } h\ge 0,\\
e^{-rt}(h_t-a(h_t)\Delta t)X_t & \text{ if } h_t<0,
\end{array}
\right.
\end{equation}
for $t=0,\dots,T-1$, incorporating the value of the loss of gas at injection. 

The discount rate is assumed to be $10\%$.

In practice, gas prices are quoted in ``dollars per million British thermal units" (\$$/$MMBtus). We note that $1000$ MMBtus are roughly equivalent to $1$ MMcf. %, therefore in our implementation, we measure gas prices by \$$/$1000MMBtus.

The calculations in \cite{ThoDavRas} are based on the gas price model
\begin{equation}
\text{d} X_t = \alpha(\beta - X_t)\text{d} t+ \gamma X_t\text{d} B_t+ (J_t-X_t)\text{d} q_t,
\end{equation}
where $t\mapsto q_t$ is a Poisson process with intensity rate $\lambda$ and independent of the Brownian motion $B_t$. Moreover, 
$J_t$ is normally distributed with mean $\mu$ and variance $\sigma^2$ independent of $B_t$ and $q_t$. In our implementation, we rescaled the parameters of \cite{ThoDavRas} to daily time-scale: $\alpha=0.25/365$, $\beta=2.5$, $\gamma=0.2/\sqrt{365}$, $\lambda=2/365$, $\mu=64$, and $\sigma^2=4$.

\begin{rem}{\rm
Since the payoff function is piece-wise linear in $h$ and the constraints sets are bounded (uniformly in $t$) for any $K$-admissible policy $\pi$, the following bounds are satisfied for all $t\in\mathcal{T}$, $x\in\mathbb{R}^+$, $y\in[0,y_{\max}]$.
\begin{multline*}
-\infty <  (-c_{\min}(0)-a(-1))\Delta t\sum_{s=t}^T\mathbb{E}[X_s|X_t=x]\\
\le
V^{\pi}_t(y,x)
\le 
c_{\max}(y_{\max})\Delta t\sum_{s=t}^T\mathbb{E}[X_s|X_t=x] < \infty.
\end{multline*}
These inequalities imply that the value function is well defined, and the dynamic programming principle holds for this particular formulation of the gas storage problem. }
\end{rem}

%*************************************************************
\subsection{The a priori estimate}\label{_sec_results_APrioriEstimate}
%*************************************************************
We computed the a priori estimate as follows. 

First, we ran the method using an equidistant initial grid $\mathcal{G}^x_0$ in the price region $[0,12]$ of interest (\cite{ThoDavRas} presents results in this price interval). However, we found that the absolute value of the second derivative of $V_0(\cdot,\dot)$ with respect to gas price was large in the price interval $[5,7]$, and close to zero otherwise; therefore, we decided to refine the grid in the middle region. In particular, we chose an initial grid $\mathcal{G}^x_0$ that had $2500$ equidistant points in the interval $[0,5]$, $5000$ equidistant points in $[5,7]$, and $2500$ equidistant points on $[7,12]$. 

\begin{table}
\begin{center}
\begin{tabular}{c|c|c|c|}
& \multicolumn{3}{c|}{Range of} \\
$X_0$ & $V_0^{\uparrow}(y_0,X_0)-V_0^{\downarrow}(y_0,X_0)$ & $\text{stdev}(V_0^{\uparrow}(y_0,X_0))$ & $\text{stdev}(V_0^{\downarrow}(y_0,X_0))$ \\
\hline 

$3$ & $[1.224,3.781]$ & $[0.115,0.121]$ & $[0.188,0.208]$ \\

$6$ & $[1.758,3.677]$ & $[0.116,0.128]$ & $[0.121,0.133]$ \\

$9$ & $[2.174,4.276]$ & $[0.060,0.076]$ & $[0.115,0.118]$

\end{tabular}
\end{center}
\caption{Comparison of high-biased and low-biased estimates: ranges of differences and ranges of estimated standard deviation over the domain $y_0\in[0,20]$ measured in \$$/$MMBtus.}
\label{tab:dualVsFwd}
\end{table}

\begin{figure}[h!]
\centering
\includegraphics[trim = 28mm 105mm 28mm 102mm, clip, width=0.75\textwidth]{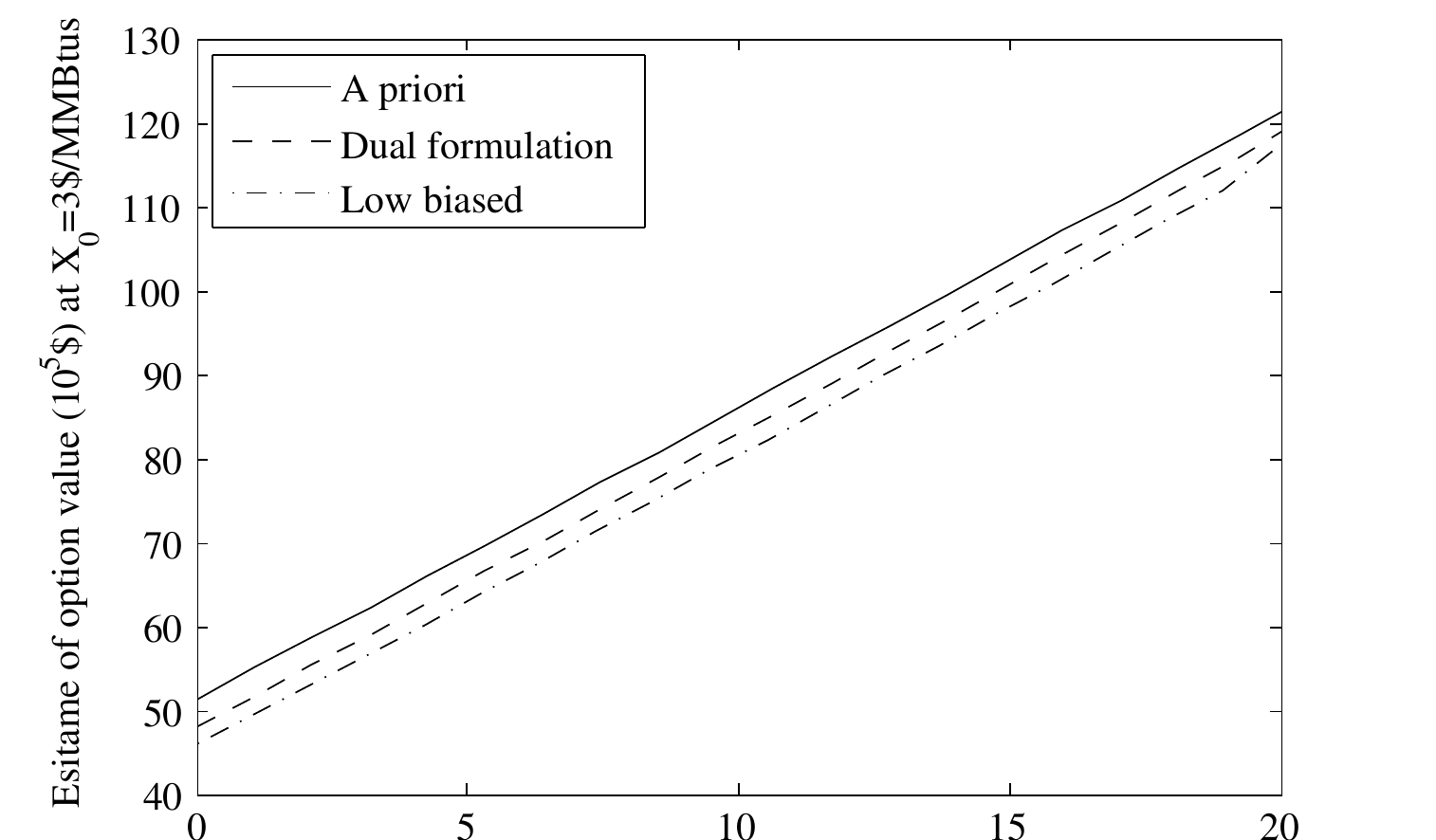}
% left bottom right top
\caption{Comparison of three methods, $X_0=3$\$$/$MMBtus.}
\label{fig:compAt3}
\end{figure}

\begin{figure}[h!]
\centering
\includegraphics[trim = 28mm 105mm 28mm 102mm, clip, width=0.75\textwidth]{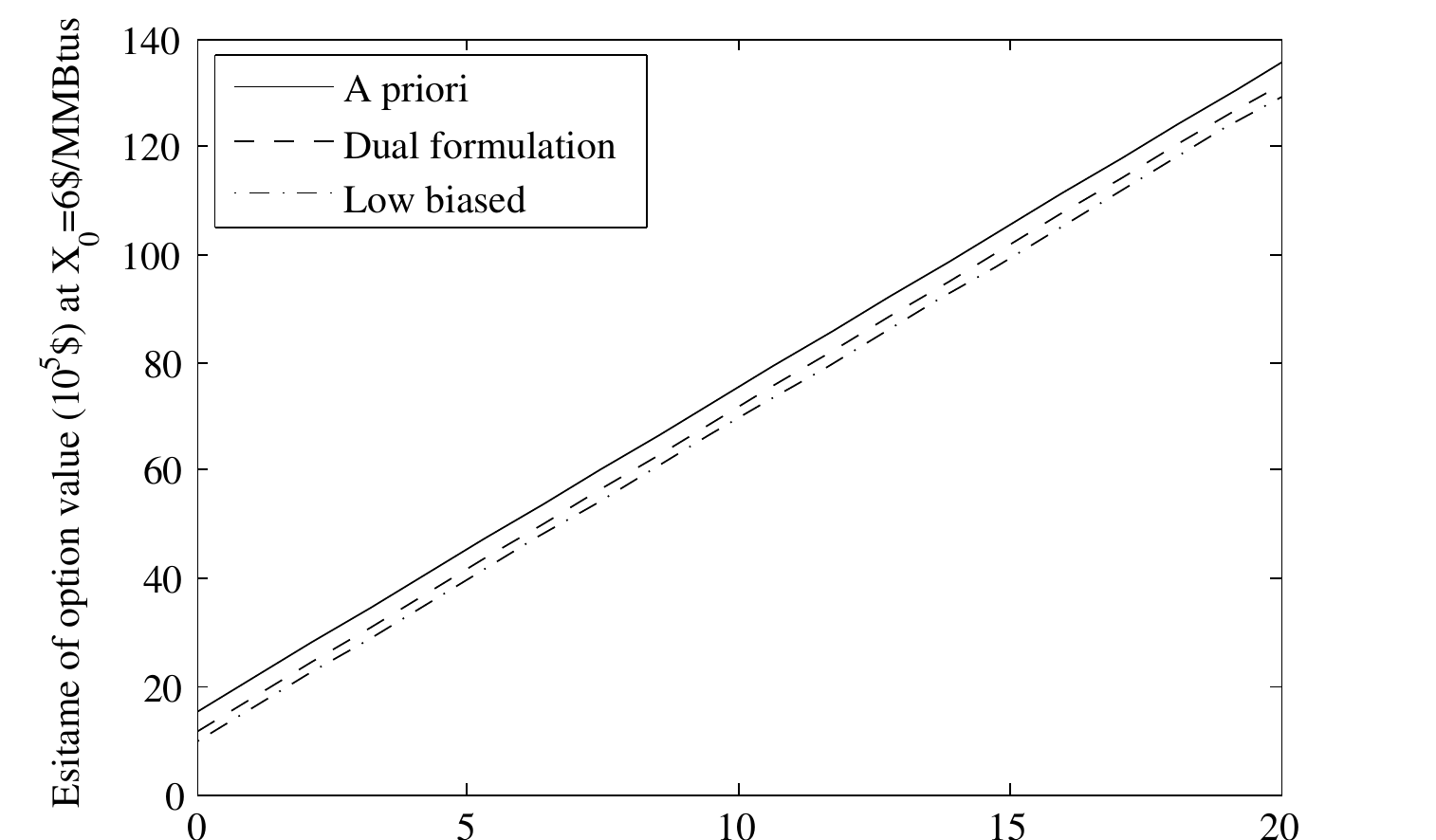}
% left bottom right top
\caption{Comparison of three methods, $X_0=6$\$$/$MMBtus.}
\label{fig:compAt6}
\end{figure}

The gas price trajectories $x^i_{\cdot}$ for $i=1,\dots, 10000$ were simulated using the Euler time-discretisation
\[
x^i_{t+1} = x^i_t+\alpha(\beta-x^i_t)\Delta t+ \gamma x^i_t \Delta B^i_t+ (J^i_t-x^i_t)\Delta q^i_t,
\]
where $\Delta B^i_t$ are independent Brownian increments on a unit time step ($\Delta t=1$), $J^i_t$ are drawn from the 
distribution of $J_t$, and $\Delta q^i_t$ drawn from the distribution
\[
\Delta q^i_t=\left\{
\begin{array}{cl}
0 & \text{ with probability } 1-\lambda \Delta t, \\
1 & \text{ with probability } \lambda \Delta t.
\end{array}
\right.
\]

In order to generate the grid $\mathcal{G}^{yx}_t$, at each time step, we generated a low discrepancy sequence 
(using a rank $1$ lattice rule with random offset, see \cite{glasserman}) of length $|\mathcal{G}^x_t|\times N_y$, and assigned $N_y$ $y$-points to each of the elements in $\mathcal{G}^x_t$. We tested the method with $N_y=3,6,21$. 

Initially, we considered using polynomial test functions for the regression. However, we found that these test functions did not capture well neither the conditional expectation function nor the value function. Therefore, we decided to use test functions that are polynomial on patches and constant outside the patches. We partitioned the $(y,x)$ domain $[0,20]\times[0,12]$ into smaller rectangles

\begin{center}
\begin{tabular}{cc} 
$[0,10]\times[0,5]$ & $[10,20]\times[0,5]$ \\
$[0,10]\times[5,7]$ & $[10,20]\times[5,7]$ \\
$[0,10]\times[7,14]$ & $[10,20]\times[7,14]$. 
\end{tabular}
\end{center}

On each rectangle, we used the following polynomials: $1$, $x$, $y$, $x^2$, $y^2$, $xy$, $x^2y$, $y^2x$, and $x^2y^2$. In addition to these polynomials, on the patches in the second row, we also used $x^3$, $x^3y$, and $x^3y^2$. Although defining functions locally on small rectangles leads to a relatively high number of test functions, the matrix \label{_eq_Bpsi} is sparse, and the evaluation is tractable. 

In step \ref{_apriori_step4} of Algorithm \ref{Main_Alg_2}, we simply compared the outcome of three scenarios: $h=0$, $h=\min K_t(y,x)$, and $h=\max K_t(y,x)$; i.e., we assumed \emph{bang-bang controls}. We also tested replacing the supremum with the maximum over finer grids in $K_t(y,x)$; however, these tests did not result in significantly different option values.

The numerical results corresponding to $t=0$, $N_y=6$, and bang-bang controls are shown in Figures \ref{fig:apriori1} and \ref{fig:apriori2}. 
Comparing these figures to the plots on page 235 in \cite{ThoDavRas}, we find that %the a priori method well estimates both the option value and the optimal production rate surfaces. We observe that the 
our a priori method slightly overestimates the option value. Given that, in order to estimate the values at time $0$, the a priori method uses information from later times, it is likely to be a \emph{high biased} method (see comments on the least squares regression based methods in \cite{glasserman}). 

%*************************************************************
\subsection{The dual upper bound}\label{_sec_results_dual}
%*************************************************************
We implemented the version of the method based on the dual formulation as specified in Section \ref{_sec_fy2} for three different initial gas prices ($3$\$$/$MMBtus, $6$\$$/$MMBtus and $9$\$$/$MMBtus). In each case, we generated $N=10000$ gas price trajectories. For $\mathcal{G}^y_t$, we used a fixed equidistant grid in $[0,20]$ with $N_y=320$ points. 

For the parametric curve fitting component, we partitioned the control interval $[0,20]$ into three shorter intervals ($[0,7]$, $[7,14]$, and $[14,20]$), and on each small interval we used the following polynomials as test functions: $1$, $y$, $y^2$, and $y^3$. 

In order to compute the optimization in step \ref{_step3} of Algorithm \ref{Main_Alg}, we approximated the supremum with the maximum on a finite grid in $K_t(y,x)$. This grid can be chosen to be finer than $\mathcal{G}^y_t$. With other optimization techniques, even more accurate estimates can be computed. 

In order to estimate the accuracy of the method, we ran the algorithm using finer $\mathcal{G}^y_t$ grids but the same set of gas price trajectories, more test functions defined locally on finer partitions, and more accurate optimization. Since the refined specifications resulted in absolute differences that were around $10\%-15\%$ of the standard deviation 
of the results, we consider the refined estimates numerically equivalent to our reference results. 

We also computed low biased estimates following the method described in Section \ref{_sec_lowBiased} using the a priori value functions and a sample of $50000$ gas price trajectories. 

The results are given in Table \ref{tab:dualVsFwd} and plotted in Figures \ref{fig:compAt3}, \ref{fig:compAt6}, and \ref{fig:compAt9}. 
For each case ($X_0=3,6,9$), Table \ref{tab:dualVsFwd} describes the range of differences of the high-biased and low-biased estimates over the range of control $y\in[0,20]$. We also provide the range of estimated standard deviations to indicate the order of magnitude of the statistical error. We note that a conservative upper and lower bound can be computed by adding three times its standard deviation to the upper estimate and subtracting three times its standard deviation from the low-biased estimate. 

Figures \ref{fig:compAt3}, \ref{fig:compAt6}, and \ref{fig:compAt9} suggest that, in some cases, the dual formulation method based estimate results in a sharper upper bound compared to the estimates of the a priori method. The upper and lower estimates are consistent with the numerical results of \cite{ThoDavRas}. 

\begin{figure}[h!]
\centering
\includegraphics[trim = 28mm 105mm 28mm 102mm, clip, width=0.75\textwidth]{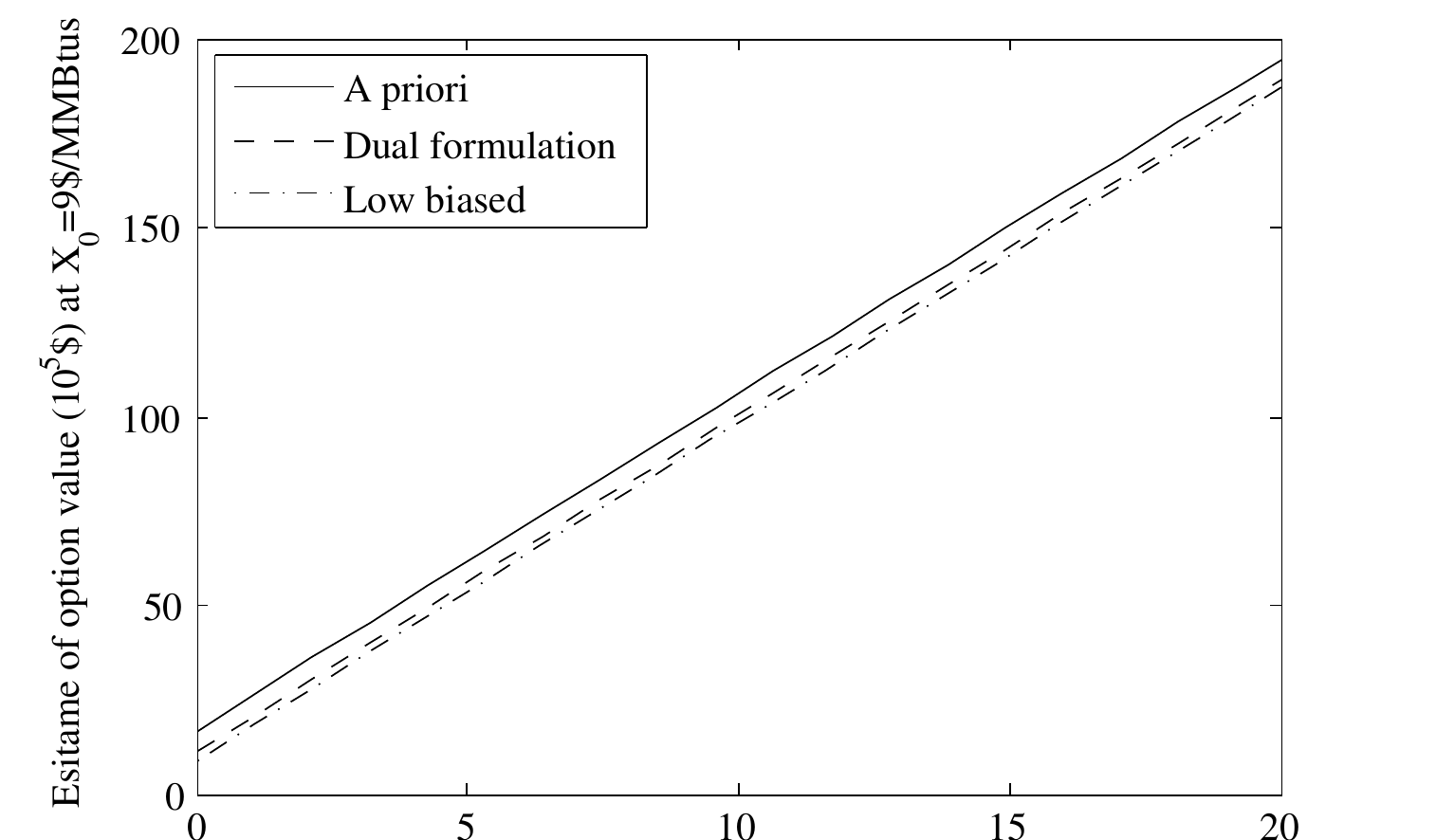}
% left bottom right top
\caption{Comparison of three methods, $X_0=9$\$$/$MMBtus.}
\label{fig:compAt9}
\end{figure}


\begin{thebibliography}{99}
%*************************************************************
%*************************************************************
\bibitem{AleHam} Aleksandrov, N., and Hambly, B.M. \textit{A dual approach to multiple exercise options 
under constraints}. Math. Methods Oper. Res. {\bf 71} (2010), 503--533.

\bibitem{AleHam2} Aleksandrov, N., and Hambly, B.M. \textit{Liquidity modelling and optimal liquidation in bond 
markets.} preprint, (2010).

\bibitem{AlmgrenChriss} Almgren, R. \& Chriss, N., \textit{Optimal execution of portfolio transactions}, Journal of Risk, 
{\bf 3} (2000/2001), 5--39.

\bibitem{BelKolSch} Belomestny, D., Kolodko, A. and Schoenmakers, J., \textit{Regression Methods for Stochastic Control Problems and Their Convergence Analysis}, SIAM J. Control Optim. {\bf 48} (2009/10),  3562--3588.

\bibitem{Ben} Bender, C. \textit{Dual pricing of multi-exercise options under volume constraints}, Finance Stoch. {\bf 15}
(2011), 1--26. 

\bibitem{Ben2} Bender, C. \textit{Primal and Dual Pricing of Multiple Exercise Options in Continuous Time}, 
SIAM J. Finan. Math., {\bf 2} (2011), 562--586.

\bibitem{BooJon} Boogert A. and Jong C., \textit{Gas Storage Valuation Using a Monte Carlo Method},
J. Derivatives, {\bf  15} (2008), 81--98. 
% preprint, \url{http://www.ems.bbk.ac.uk/research/wp/PDF/BWPEF0704.pdf},(2006)

\bibitem{ChenFors} Chen, Z. and Forsyth, P.A., \textit{A semi-Lagrangian approach for natural gas storage valuation 
and optimal operation}, SIAM J. Sci. Comput., {\bf 30} (2007), 339--368.

\bibitem{ClemLamPro} Clement, E., Lamberton, D. and Protter, P.\textit{An Analysis of a Least Squares regression 
Method for American Option Pricing}, Finance Stoch. {\bf 6} (2002), 449--471. 

\bibitem{DavKar} Davis, M. H. A. and Karatzas, I.
\textit{A deterministic approach to optimal stopping} In: Probability, statistics and optimisation, 455--466,
Wiley Ser. Probab. Math. Statist., Wiley, Chichester, (1994).

\bibitem{glasserman} Glasserman, P., \textit{Monte Carlo Methods in Financial Engineering}, Springer, (2003)

\bibitem{HauKog}
Haugh, M. B. and Kogan, L. \textit{Pricing American options: a duality approach.} Oper. Res. 52 (2004), 258--270.

\bibitem{LonSchw} Longstaff F.A and Schwartz, E.S., \textit{Valuing American Options by Simulation: A
Simple Least-Squares Approach}, Rev. Financial Studies, {\bf 14} (2001), 113--147. 

\bibitem{LudCar} Ludovski, M. \& Carmona, R., \textit{Valuation of Energy Storage: An Optimal Switching Approach} 
Quantitative Finance, {\bf 10} (2010), 359--374.

\bibitem{MeiHam} Meinshausen, N. \& Hambly, B.M., \textit{Monte Carlo methods for the valuation of multiple 
exercise options}, Math. Finance, {\bf 14} (2004), 557--583. 

\bibitem{Rog} Rogers L.C.G. \textit{Monte Carlo valuation of
american options}, Math. Finance, {\bf 12} (2002), 271--286.

\bibitem{Rog1} Rogers, L. C. G., \textit{Pathwise stochastic optimal control}, SIAM J. Control Optim., {\bf 46}  
(2007), 1116--1132.

\bibitem{torsten} Sch\"oneborn, T., \textit{Adaptive basket liquidation}, preprint 2011,\\
\url{http://papers.ssrn.com/sol3/papers.cfm?abstract_id=1343985}

\bibitem{Sch} Schonmakers, J., \textit{A pure martingale dual for multiple stopping}, Finance Stoch., Published online: 30 November 2010, DOI: 10.1007/s00780-010-0149-1

\bibitem{ThoDavRas} Thompson, M., Davison, M. and Rasmussen, H., \textit{Natural Gas Storage Valuation and Optimization: A Real Options application}, Naval Research Logistics, {\bf 56} (2009), 226--238.

\bibitem{VanRoyTsits}  Tsitsiklis J. N. and Van Roy B., \textit{Regression Methods for Pricing Complex American-Style Options}, IEEE Trans. on Neural Networks, {\bf 12} (2001), 694--703.

\end{thebibliography}
\end{document}